\documentclass{article}
\usepackage[utf8]{inputenc}
\usepackage[margin=1in]{geometry}
\usepackage{amsmath, amssymb, amsthm} 
\usepackage[nolist]{acronym}

\usepackage{amssymb}
\usepackage{mathtools}
\usepackage{mathrsfs}
\usepackage{subcaption}
\usepackage{tikz}
\usepackage{hyperref}

\hypersetup{colorlinks=true, linkcolor=black}

\newcommand{\R}{\mathbb{R}} 
\newcommand{\K}{\mathbb{K}} 
 
\newcommand{\N}{\mathbb{N}}

\newcommand{\Lie}{\mathcal{L}}

\newcommand{\M}{\mathbb{M}}
\newcommand{\scL}{\mathscr{L}}

\newcommand{\ts}{\mathcal{T}}
\newcommand{\gs}{\mathcal{G}}
\newcommand{\es}{\mathcal{E}}
\newcommand{\vs}{\mathcal{V}}

\newcommand{\A}{\mathcal{A}}

\newcommand{\bm}{\mathbf{m}}
\newcommand{\rbm}{\rev{\bm}}
\newcommand{\psd}{\mathbb{S}}

\DeclarePairedDelimiter{\abs}{\lvert}{\rvert}
\DeclarePairedDelimiter{\norm}{\lVert}{\rVert}
\DeclarePairedDelimiter{\ceil}{\lceil}{\rceil}

\DeclarePairedDelimiter{\diag}{\textrm{diag}(}{)}

\DeclarePairedDelimiterX{\inp}[2]{\langle}{\rangle}{#1, #2}

\DeclarePairedDelimiter{\Mp}{\mathcal{M}_+(}{)}

\newtheorem{theorem}{Theorem}[section]
\newtheorem{lem}{Lemma}[section]

\newtheorem{cor}{Corollary}[section]

\newtheorem{remark}{Remark}
\newtheorem{problem}{Problem}
\newcommand{\rev}[1]{#1}
\newcommand{\rw}[1]{ #1}
\title{\LARGE \bf
Bounding the Distance to Unsafe Sets with \\ Convex Optimization
}
\renewcommand\footnotemark{}

\author{Jared Miller$^1$, Mario Sznaier$^1$
\thanks{$^1$J. Miller and M. Sznaier are with the Robust Systems Lab,  ECE Department, Northeastern University, Boston, MA 02115. (e-mails: miller.jare@northeastern.edu,  msznaier@coe.neu.edu).}
\thanks{ J. Miller and M. Sznaier were partially supported by NSF grants  CNS--1646121, CMMI--1638234, ECCS--1808381 and CNS--2038493,  and AFOSR grant FA9550-19-1-0005. 
This material is based upon research supported by the Chateaubriand Fellowship of the Office for Science \& Technology of the Embassy of France in the United States.
}
}
\date{\empty}
\begin{document}

\maketitle

\begin{abstract}
\label{sec:abstract}
This work proposes an algorithm to bound the minimum distance between points on trajectories of a dynamical system and points on an unsafe set. Prior work on certifying safety of trajectories includes barrier and density methods, which do not provide a margin of proximity to the unsafe set in terms of distance. The distance estimation problem is relaxed to a Monge-Kantorovich type optimal transport problem based on existing occupation-measure methods of peak estimation. Specialized programs may be developed for polyhedral norm distances (e.g. L1 and Linfinity) and for scenarios where a shape is traveling along trajectories (e.g. rigid body motion). The distance estimation problem will be correlatively sparse when the distance objective is separable.

\end{abstract}
\section{Introduction}
\label{sec:intro}

A trajectory is safe with respect to an unsafe set $X_u$ if no point along the trajectory contacts or enters $X_u$. Safety of trajectories may be quantified by the distance of closest approach to $X_u$, which will be positive for all safe trajectories and zero for all unsafe trajectories. The task of finding this distance of closest approach will also be referred to as `distance estimation'.
In this setting, an agent with state $x$ is restricted to a state space $X \subseteq \R^n$ and starts in an initial set $X_0 \subset X$. The trajectory of an agent evolving according to locally Lipschitz dynamics $\dot{x} = f(t, x(t))$ starting at an initial condition $x_0 \in X_0$ is denoted by $x(t \mid x_0)$.
The closest approach as measured by a distance function $c$ that any trajectory takes to the unsafe set $X_u$ in a time horizon of $t \in [0, T]$ can be found by solving, 
\begin{equation}
    \label{eq:dist_traj}
    \begin{aligned}
    P^* = & \rev{\inf}_{t,\:x_0, y} c(x(t \mid x_0), y) \\
    & \dot{x}(t) = f(t, x), \quad t \in [0, T] \\
    & \rev{x(0) = }x_0 \in X_0, \ y \in X_u.
    \end{aligned}
\end{equation}
Solving \eqref{eq:dist_traj} requires optimizing over all points $(t, x_0, y) \in [0, T] \times X_0 \times X_u$, which is generically a non-convex and difficult task.
Upper bounds to $P^*$ may be found by sampling points $(x_0, y)$ and evaluating $c(x(t \mid x_0), y)$ along these sampled trajectories. Lower bounds to $P^*$ are a universal property of all trajectories, and will satisfy $P^* > 0$ if all trajectories starting from $X_0$ in the time horizon $[0, T]$ are safe with respect to $X_u$. 

This paper proposes an occupation-measure based method to compute lower bounds of $P^*$ through a converging hierarchy of \rev{convex \acp{SDP}} \cite{boyd1994linear}. 
These \acp{SDP} arise from \rw{a} finite truncation of infinite dimensional \acp{LP} in measures \cite{lasserre2009moments}. Occupation measures are Borel measures that contain information about the distribution of states evolving along trajectories of a dynamical system. The distance estimation \ac{LP} formulation is based on measure \acp{LP} arising from  peak estimation of dynamical systems \cite{helmes2001computing, cho2002linear, fantuzzi2020bounding} because the state function to be minimized along trajectories is the point-set distance function between $x\in X$ and $X_u$. 
Inspired by optimal transport theory \cite{villani2008optimal, santambrogio2015optimal, peyre2019computational}, the distance function $c(x, y)$ between points $x \in X$ on trajectories and $y \in X_u$ is relaxed to an \rev{expectation of the distance $c(x, y)$ with respect to} probability distributions over $X$ and $X_u$. 

Occupation measure \acp{LP} for control problems were first formulated in \cite{lewis1980relaxation}, and their \ac{LMI} relaxations were detailed in \cite{henrion2008nonlinear}. These occupation measure methods have also been applied to region of attraction estimation and backwards reachable set maximizing control \cite{henrion2013convex, korda2013inner, korda2014convex}. 

Prior work on verifying safety of trajectories includes  Barrier functions \cite{prajna2004safety, prajna2006barrier},  Density functions \cite{rantzer2004analysis}, and Safety Margins \cite{miller2020recovery}. Barrier and Density functions offer binary indications of safety/unsafety; if a Barrier/Density function exists, then all trajectories starting from $X_0$ are safe. Barrier/Density functions may be non-unique, and the existence of such a function does not yield a measure of closeness to the unsafe set. Safety Margins are a measure of constraint violation, and a negative safety margin verifies safety of trajectories. Safety Margins can vary with constraint reparameterization, \rev{even in the same coordinate system} (e.g. multiplying all defining constraints of $X_u$ by a positive constant \rev{scales the safety margin by that constant}), and therefore yield a qualitative certificate of safety. The distance of closest approach $P^*$ is independent of constraint reparameterization\rev{, returning} quantifiable and geometrically interpretable information about safety of trajectories. 


The contributions of this paper include:
\begin{itemize}
    \item A measure \ac{LP} to lower bound the distance estimation task \eqref{eq:dist_traj}
    \item A proof of convergence to $P^*$ within arbitrary accuracy as the degree of \ac{LMI} approximations approaches infinity
        \item A decomposition of the distance estimation LP using correlative sparsity when the cost $c(x, y)$ is separable 
    \item Extensions such as finding the distance of closest approach between a shape with evolving orientations and the unsafe set
\end{itemize}

Parts of this paper were \rw{presented at} the 61st Conference on Decision and Control
\cite{miller2022distance_short}. Contributions of this work over and above the conference version include:
\begin{itemize}
    \item A discussion of the scaling properties of safety margins 
    \item An application of correlative sparsity in order to reduce the computational cost of finding distance estimates
    \item An extension to bounding the set-set distance between a moving shape and the unsafe set
    \item A presentation of a lifting framework for polyhedral norm distance functions
    \item A full proof of strong duality
\end{itemize}
%

This paper is structured as follows: Section \ref{sec:prelim} reviews preliminaries such as notation and measures for peak and safety estimation. Section \ref{sec:distance_lp} proposes an infinite-dimensional \ac{LP} to bound the distance closest approach between points along trajectories and points on the unsafe set.  
\rev{Section \ref{sec:distance_finite} truncates the infinite-dimensional \acp{LP} into \acp{SDP} through the moment-\ac{SOS} hierarchy, and studies numerical considerations associated with these \acp{SDP}}.
Section \ref{sec:sparsity} utilizes correlative sparsity to create \ac{SDP} relaxations of distance estimation with smaller \ac{PSD} matrix constraints.
Distance estimation problems for shapes traveling along trajectories are posed in Section \ref{sec:shape}.
Examples of the distance estimation problem are presented in Section \ref{sec:examples}.
Section \ref{sec:extensions} details extensions to the distance estimation problem, including uncertainty, polyhedral norm distances, and application of correlative sparsity.  
The paper is concluded in Section \ref{sec:conclusion}. Appendix \ref{app:duality} offers a proof of strong duality betwen infinite-dimensional LPs for distance estimation. \rev{Appendix \ref{sec:moment_sos} summarizes the moment-SOS hierarchy}.
\begin{acronym}[WSOS]

\acro{BSA}{Basic Semialgebraic}


\acro{CSP}{Correlative Sparsity Pattern}

\acro{LMI}{Linear Matrix Inequality}
\acroplural{LMI}[LMIs]{Linear Matrix Inequalities}
\acroindefinite{LMI}{an}{a}


\acro{LP}{Linear Program}
\acroindefinite{LP}{an}{a}


\acro{POP}{Polynomial Optimization Problem}

\acro{PSD}{Positive Semidefinite}

\acro{PD}{Positive Definite}

\acro{SDP}{Semidefinite Program}
\acroindefinite{SDP}{an}{a}

\acro{SOS}{Sum of Squares}
\acroindefinite{SOS}{an}{a}

\acro{WSOS}{Weighted Sum of Squares}

\end{acronym}

\section{Preliminaries}
\label{sec:prelim}
\subsection{Notation and Measure Theory}



Let $\R$ be the set of real numbers and $\R^n$ be an $n$-dimensional real Euclidean space. Let $\N$ be the set of natural numbers and $\N^n$ be the set of $n$-dimensional multi-indices. The total degree of a multi-index $\alpha \in \N^n$ is $\abs{\alpha} = \sum_i \alpha_i$. A monomial $\prod_{i=1}^n x_i^{\alpha_i}$ may be expressed in multi-index notation as $x^\alpha$. The set of polynomials with real coefficients is $\R[x]$, and polynomials $p(x) \in \R[x]$ may be represented as the sum over a finite index set $\mathscr{J} \subset \N^n$ of $p(x) = \sum_{\alpha \in \mathscr{J}} p_\alpha x^\alpha$.
The set of polynomials with monomials up to degree $\abs{\alpha} = d$ is $\R[x]_{\leq d}$.
A metric function $c(x, y)$ over the space $X$ with $x, y \in X$ satisfies the following properties \cite{deza2009encyclopedia}:
\begin{subequations}
\begin{align}
    &c(x, y) = c(y, x) > 0 & x \neq y  \\
    &c(x, x) = 0& \\
    &c(x, y) \leq c(x, z)  + c(z ,y) & \forall z \in \rw{S}.
\end{align}
\end{subequations}
The set of metrics are closed under addition and pointwise maximums. 
Every norm $\norm{\cdot}$ inspires a metric $c_{\norm{\cdot}}(x, y) = \norm{x-y}$.
The point-set distance function $c(x; Y)$ between a point $x \in X$ and a closed set $Y \subset X$ is defined by:
\begin{equation}
    c(x; Y) = \inf_{y \in Y} c(x, y).
\end{equation}
 

The set of continuous functions over \rw{a Banach space} $\rw{S}$ is denoted as $C(\rw{S})$, the set of finite signed Borel measures over $\rw{S}$ is $\mathcal{M}(\rw{S})$, and the set of nonnegative Borel measures over $\rw{S}$ is $\Mp{\rw{S}}$. A duality pairing exists between all functions $f \in C(\rw{S})$ and measures $\mu \in \Mp{\rw{S}}$ by Lebegue integration: $\inp{f}{\mu} = \int_{\rw{S}} f(s) d \mu(s)$ \rev{when $\rw{S}$ is compact}. The subcone of nonnegative continuous functions over $\rw{S}$ is $C_+(\rw{S}) \subset C(\rw{S})$, which satisfies $\inp{f}{\mu} \geq 0 \ \forall f \in C_+(\rw{S}), \ \mu \in \Mp{\rw{S}}$. 
The subcone of continuous functions over $\rw{S}$ \rev{whose first $k$ derivatives are continuous is $C^k(\rw{S})$ (with $C(\rw{S}) = C^0(\rw{S})$)}.
The indicator function of a set $A \subseteq \rw{S}$ is a function $I_A: \rw{S} \rightarrow \{0, 1\}$ taking values $I_A(\rw{s}) = 1$ if $\rw{s} \in A$ and $I_A(\rw{s}) = 0$ if $\rw{s} \not\in A$. The measure of a set $A$ with respect to $\mu \in \Mp{\rw{S}}$ is $\mu(A) = \inp{I_A(\rw{s})}{\mu} = \int_A d \mu$. The mass of $\mu$ is $\mu(\rw{S}) = \inp{1}{\mu}$, and $\mu$ is a probability measure if $\inp{1}{\mu} = 1$. The support of $\mu$ is the set of all points $\rw{s} \in \rw{S}$ such that every open neighborhood $N_\rw{s}$ of $\rw{s}$ has $\mu(N_x) > 0$. \rev{The Lebesgue measure $\lambda_{\rw{S}}$ over a space $\rw{S}$ is the volume measure satisfying $\inp{f}{\lambda_{\rw{S}}} = \int_{\rw{S}} f(\rw{s}) d\rw{s}\ \forall f \in C(\rw{S})$.} The Dirac delta $\delta_{\rw{s}'}$ is a probability measure supported at a single point $\rw{s}' \in \rw{S}$, and the duality pairing of any function $f \in C(\rw{S})$ with respect to $\delta_{\rw{s}'}$ is $\inp{f(\rw{s})}{\delta_{\rw{s}'}} = f(\rw{s}')$. A measure $\mu = \sum_{i=1}^r c_i \delta_{\rw{s}_i}$ that is the conic combination (weights $c_i > 0$) of $r$ distinct Dirac deltas is known as a rank-$r$ atomic measure. The atoms of $\mu$ are the support points $\{\rw{s}_i\}_{i=1}^r$.

Let $\rw{S}, Y$ be spaces and $\mu \in \Mp{\rw{S}}, \ \nu \in \Mp{Y}$ be measures. The product measure $\mu \otimes \nu$ is the unique measure such that  $\forall A \in \rw{S}, \ B \in Y: \ (\mu \otimes \nu)(A \times B) = \mu(A) \nu(B)$. 
The pushforward of a map $Q: \rw{S}\rightarrow Y$ along a measure $\mu\rev{(\rw{s})}$ is $Q_\# \mu\rev{(y)}$, which satisfies 
$\forall f \in C(Y): \ \inp{f(y)}{Q_\# \mu(y)} = \inp{f(Q(\rw{s}))}{\mu(\rw{s})}$. The measure of a set $B \in Y$ is $Q_\# \mu (Y) = \mu(Q^{-1} (Y))$. The projection map 
$\pi^\rw{s}: \rw{S} \times Y \rightarrow \rw{S}$ preserves only the
$\rw{s}$-coordinate as $(\rw{s}, y) \rightarrow \rw{s}$ and a similar definition holds for $\pi^y$. Given a measure $\eta \in \Mp{\rw{S} \times Y}$, the projection-pushforward $\pi^{\rw{s}}_\# \eta$ expresses the $\rw{s}$-marginal of $\eta$ with duality pairing $\forall f \in C(\rw{S}): \inp{f(\rw{s})}{\pi_\#^\rw{s} \eta} = \int_{\rw{S} \times Y} f(\rw{s}) d \eta(\rw{s}, y) $.
Every linear operator $\Lie: \rw{S} \rightarrow Y$ possesses a unique adjoint operator $\Lie^\dagger$ such that $\inp{\Lie f}{\mu} = \inp{f}{\Lie^\dagger \mu}, \ \forall f \in C(\rw{S}), \mu \in \Mp{\rw{S}}$.

\rw{Given an interval $[a, b]$ and a continuous curve $\rw{s}(t)$ where $\rw{s} : [a, b] \rightarrow \rw{S}$ and \rw{$S \subset \R^n$},
the pushforward of the Lebesgue measure on $[a, b]$ through the map $t \rightarrow  (t, \rw{s}(t))$ is called the
occupation measure associated to $\rw{s}(t)$. The sup-norm of a function $f \in C^0(\rw{S})$ is $\norm{f}_{C^0(\rw{S})} = \sup_{\rw{s} \in \rw{S}} \abs{f(\rw{s})}$. The $C^1$ norm of a function $f\in C^1(\rw{S})$ is $\norm{f}_{C^1(\rw{S})} = \norm{f}_{C^0(\rw{S})} + \sum_{i=1}^n \norm{\partial_{s_i} f}_{C^0(\rw{S})}$.}

\subsection{Peak Estimation and Occupation Measures}
\label{sec:peak}


The peak estimation problem involves finding the maximum value of a  function $p(x)$ along trajectories of a dynamical system \rw{starting from $X_0 \subset X \subset \R^n$ and remaining in $X$}, 
\begin{equation}
    \begin{aligned}
    P^* = & \rev{\sup}_{t \in [0, T],\, x_0 \in X_0} p(x(t \mid x_0)),  \\
    &\label{eq:peak_traj_std} \rw{x(0) = x_0, } \quad \dot{x}(t) = f(t, x(t)).
    \end{aligned}
\end{equation}

Every optimal trajectory of \eqref{eq:peak_traj_std}  \rev{(if one exists)} may be described by a tuple $(x_0^*, t_p^*)$ with $x_p^* = x(t^*_p \mid x_0^*)$ satisfying $P^* = p(x^*_p) = p(x(t^*_p \mid x_0^*))$. 
A persistent example throughout this paper will be the Flow system of \cite{prajna2004safety}:
\begin{equation}
\label{eq:flow}
    \dot{x} = \begin{bmatrix}x_2 \\ -x_1 -x_2 + \frac{1}{3}x^3_1\end{bmatrix}.
\end{equation}

Figure \ref{fig:flow_recovery} plots trajectories of the flow system in cyan for times $t \in [0, 5]$, starting from the initial set $X_0 =  \{x \mid (x_1-1.5)^2 + x_2 \leq 0.4^2\}$ in the black circle. The minimum value of $x_2$ along these trajectories is  $\min x_2 \approx -0.5734$. The optimizing trajectory is shown in dark blue, starting at the blue circle $x_0^* \rev{\approx} (1.4889, -0.3998)$ and reaching optimality at $x_p^* \rev{\approx} (0.6767, -0.5734)$ in time $t^*_p\rev{\approx} 1.6627$.

    \begin{figure}[ht]
        \centering
        \includegraphics[width=0.5\linewidth]{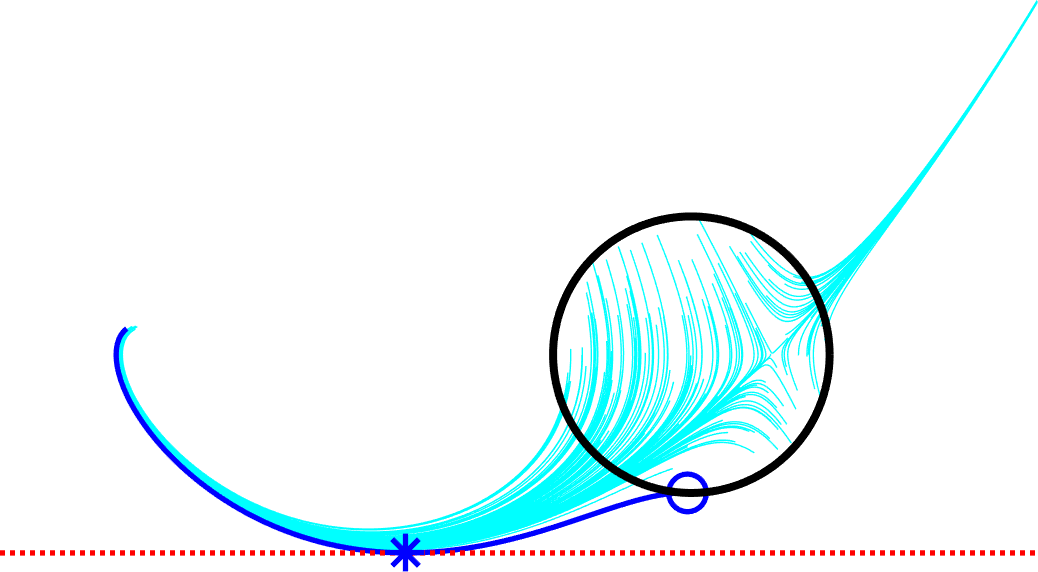}
        \caption{\label{fig:flow_recovery} Minimizing $x_2$ along Flow system \eqref{eq:flow}}
    \end{figure}

The work in \cite{cho2002linear} developed a measure \ac{LP} to find an upper bound $p^* \geq P^*$. This measure \ac{LP} involves an initial measure $\mu_0 \in \Mp{X_0}$, a peak measure $\mu_p \in \Mp{[0, T] \times X}$, and an occupation measure $\mu \in \Mp{[0, T] \times X}$ connecting together $\mu_0$ and $\mu_p$. Given a distribution of initial conditions $\mu_0 \in \Mp{X_0}$ and a stopping time $0 \leq t^* \leq T$, the occupation measure $\mu$ of a set $A \times B$ with $A \subseteq [0, T], \ B \subseteq X$ is defined by,
\begin{equation}
    \label{eq:avg_free_occ}
    \mu(A \times B) = \int_{[0, t^*] \times X_0} I_{A \times B}\left((t, x(t \mid x_0))\right) dt \, d\mu_0(x_0).
\end{equation}
The measure $\mu(A \times B)$ is the $\mu_0$-averaged amount of time a trajectory will dwell in the box $A \times B$. With ODE dynamics $\dot{x}(t) = f(t, x(t))$, the Lie derivative $\Lie_f$ along a test function $v \in C^1([0, T]\times X)$ is,
\begin{equation}
    \Lie_f v(t, x) = \partial_t v(t,x) + f(t,x) \cdot \nabla_x v(t,x).
\end{equation}

Liouville's equation expresses the constraint that 
$\mu_0$ is connected to $\mu_p$ by trajectories with dynamics $f$ for all test functions $v \in  C^1([0, T]\times X)$,
\begin{align}
\inp{v(t, x)}{\mu_p} &= \inp{v(0, x)}{\mu_0} + \inp{\Lie_f v(t,x)}{\mu} \label{eq:liou_int}\\
\mu_p &= \delta_0 \otimes \mu_0 + \Lie_f^\dagger \mu. \label{eq:liou_weak}
\end{align}

Equation \eqref{eq:liou_weak} is an equivalent short-hand expression to equation \eqref{eq:liou_int} for all $v$. Substituting in the test functions $v = 1, v = t$ to Liouville's equation returns the relations $\inp{1}{\mu_0} = \inp{1}{\mu_p}$ and $\inp{1}{\mu} = \inp{t}{\mu_p}$.

The measure \ac{LP} corresponding to \eqref{eq:peak_traj_std} with optimization variables $(\mu_0, \mu_p, \mu)$ is \cite{cho2002linear},
\begin{subequations}
\label{eq:peak_meas}
\begin{align}
p^* = & \ \rev{\sup} \quad \inp{p(x)}{\mu_p} \label{eq:peak_meas_obj} \\
    & \mu_p = \delta_0 \otimes\mu_0 + \Lie_f^\dagger \mu \label{eq:peak_meas_flow}\\
    & \inp{1}{\mu_0} = 1 \label{eq:peak_meas_prob}\\
    & \mu, \mu_p \in \Mp{[0, T] \times X} \label{eq:peak_meas_peak}\\
    & \mu_0 \in \Mp{X_0}. \label{eq:peak_meas_init}
\end{align}
\end{subequations}

Both $\mu_0$ and $\mu_p$ are probability measures by constraint \eqref{eq:peak_meas_prob}. 
\rw{A set of measures $(\mu_0, \mu_p, \mu)$ may be derived from a trajectory with initial condition $x_0^* \in X_0$ and a stopping time $t^*_p$ in which $x_p^* = x(t^*\mid x_0^*),$ $p(x_p^*) = P^*$, and  $x(t\mid x_0^*) \in X \ \forall t \in [0, t^*_p]$. }
The \rw{atomic measures are} $\mu_0 = \delta_{x = x_0^*}, \ \mu_p = \delta_{t = t^*_p} \rw{\otimes \delta_{x = x_p^*}}$ \rw{and $\mu$ is the occupation measure in times $[0, t_p^*]$ along $t \mapsto (t, x^*(t \mid x_0^*)).$} 
\rw{These measures} are solutions to constraints \eqref{eq:peak_meas_flow}-\eqref{eq:peak_meas_init}, \rw{which implies that} $p^* \geq P^*$. 
 There \rw{is no} relaxation gap ($p^* = P^*$) if the set $[0, T] \times X$ is compact \rw{with $X_0 \subset X$} (Sec. 2.3 of \cite{fantuzzi2020bounding} and \cite{lewis1980relaxation}). The moment-\ac{SOS} hierarchy \rw{of \acp{SDP}} may be used to find a sequence of upper bounds to $p^*$. The method in \cite{fantuzzi2020bounding} approaches the moment-\ac{SOS} hierarchy from the dual side, involving \ac{SOS} constraints in terms of an auxiliary function $v(t, x)$ \rev{(dual variable to constraint \eqref{eq:peak_meas_flow})}. 
 \rw{Near-optimal trajectory extraction can be attempted through \ac{SDP} solution matrix factorization  \cite{miller2020recovery} (if a low rank condition holds) and through sublevel set methods \cite{fantuzzi2020bounding, garcia2021superresolution}.}

\subsection{Safety}

This subsection reviews methods to verify that trajectories starting from $X_0 \subset{X} $ do not enter an unsafe  set  $X_u \subset X$. 
In Figure \ref{fig:flow_unsafe}, the unsafe set $X_u =  \{x \in \R^2 \mid x_1^2 + (x_2+0.7)^2 \leq 0.5^2, \ \sqrt{2}/2 (x_1 + x_2 - 0.7) \leq 0\}$ is the red half-circle to the bottom-left of trajectories.

\begin{figure}[ht]
        \centering
        \includegraphics[width=0.5\linewidth]{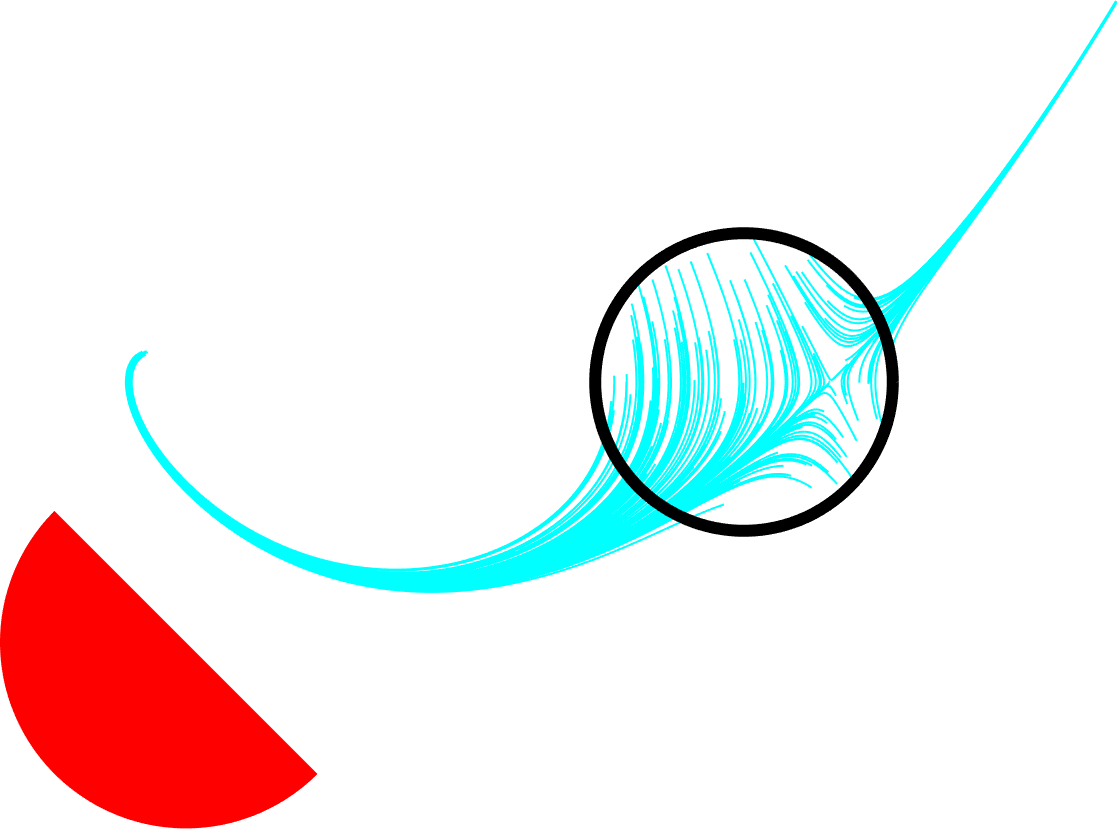}
        \caption{\label{fig:flow_unsafe} Trajectories of Flow system \eqref{eq:flow}}
        
    \end{figure}

Sufficient conditions certifying safety can be obtained using  barrier functions \cite{prajna2004safety, prajna2006barrier}. However, these conditions do not provide a quantitative measurement for the safety of trajectories. Safety margins as introduced in \cite{miller2020recovery} quantify the safety of trajectories through the use of maximin peak estimation. Assume that $X_u$ is a basic semialgebraic set with description $X_u = \{x \mid p_i(x) \geq 0, \ i = 1, \ldots, N_u\}$. A point $x$ is in $X_u$ if all $p_i(x) \geq 0$. If at least one $p_i(x)$ remains negative for all points along trajectories $x(t \mid x_0), \ x_0 \in X_0$, then no point starting from $X_0$ enters $X_u$, and trajectories  are therefore  safe. The value $p^* =\max_i \min_{t, x_0} p_i(x(t \mid x_0))$ is called the safety margin, and a negative safety margin $p^* < 0$ certifies safety. The moment-\ac{SOS} hierarchy \rev{(Appendix \ref{sec:moment_sos})} can be used to find upper bounds $p^*_d > p^*$ at degrees $d$, and safety is assured if any upper bound is negative $0 > p^*_d > p^*$. Figure \ref{fig:flow_safety_margin} visualizes the safety margin for the Flow system \eqref{eq:flow}, where the bound of $p^* \leq -0.2831$ was found at the degree-$4$ relaxation.

        \begin{figure}[ht]
        \centering
        \includegraphics[width=0.5\linewidth]{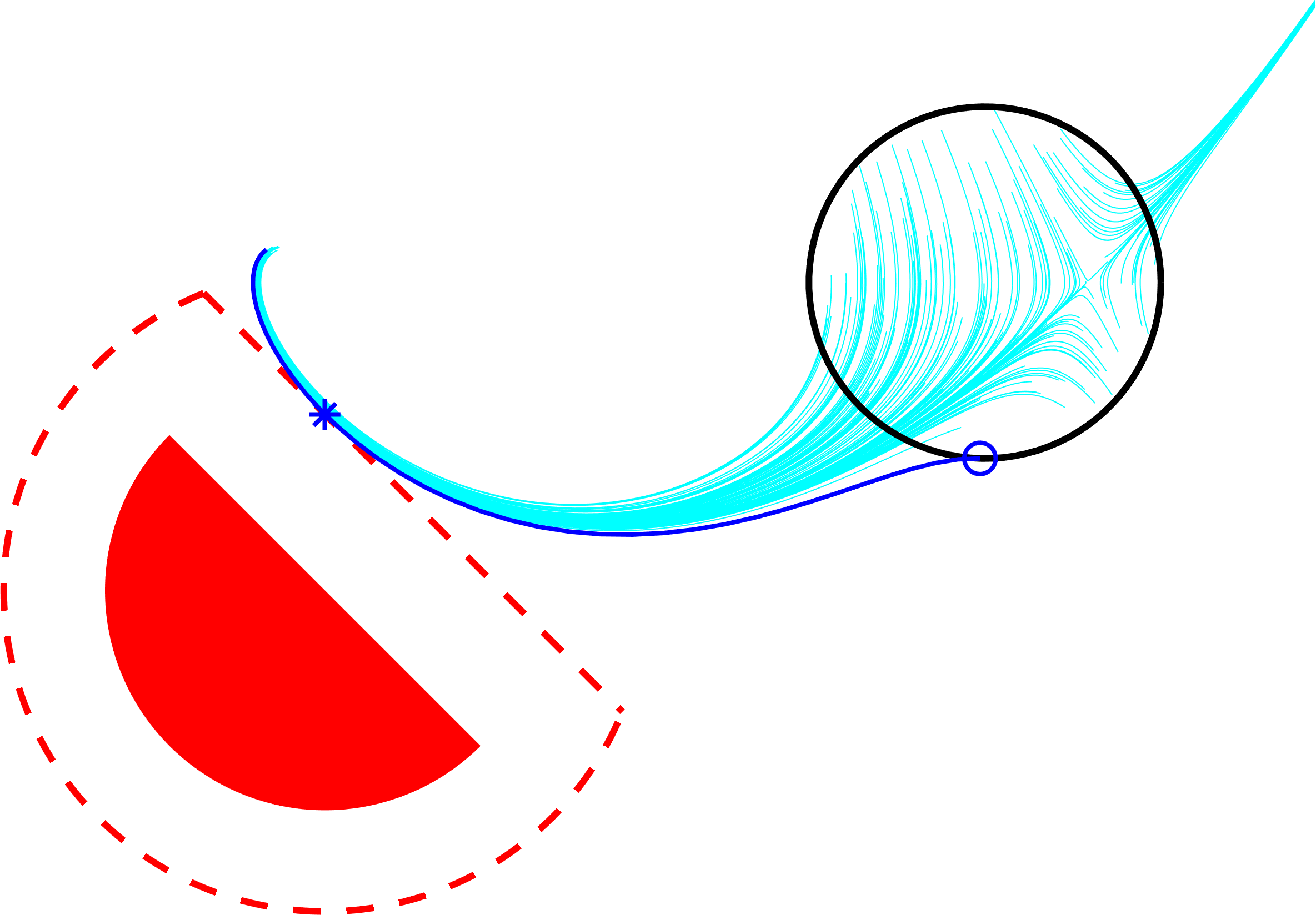}
        \caption{\label{fig:flow_safety_margin} Flow system is safe,  $p^* \leq -0.2831$}
    \end{figure}

The safety margin of trajectories will generally change if the unsafe set $X_u$ is reparameterized even in the same coordinate system. Let $q \leq 0$ and $s > 0$ be violation and scaling parameters for the enlarged unsafe set $(X_u^{s})_q = \{x \mid  q\leq 0.5^2 - x_1^2 + (x_2+0.7)^2, \ q \leq -s (x_1 + x_2 - 0.7) \}$. The original unsafe set may be interpreted as $X_u = (X_u^{s=\sqrt{2}/2})_{q=0}$.
Figure \ref{fig:safety_scaling} visualizes contours of regions $(X_u^s)_q$ 
as $q$ decreases from $0$ down to $-2$ for sets with scaling parameters $s = 5$ and $s = 1$. The safety margins of trajectories with respect to $X_u^{s}$ 
will vary as $s$ changes, 
even as the same set $X_u$ is represented in all cases. This is precisely the difficulty addressed in the present paper: developing scale invariant quantitative safety metrics.

    \begin{figure}[ht]
        \centering
        \includegraphics[width=\linewidth]{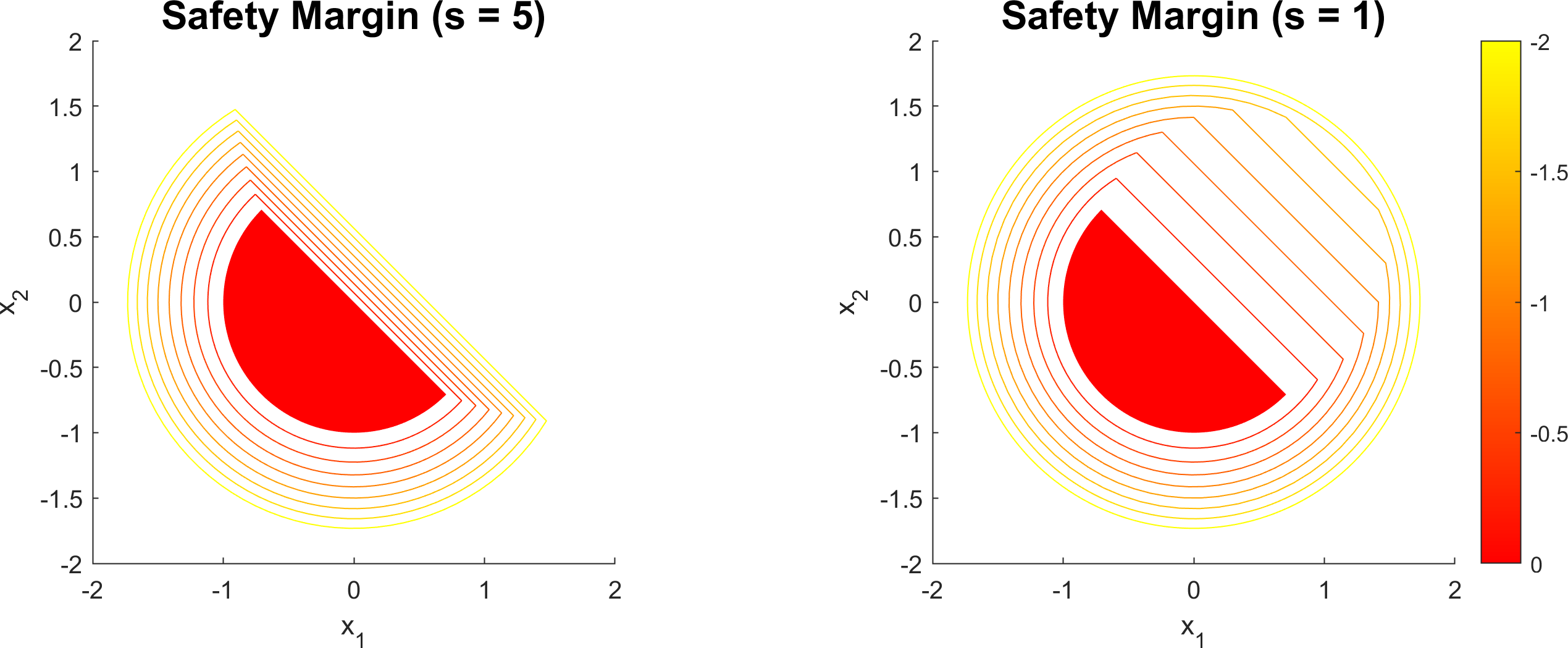}
        \caption{\label{fig:safety_scaling} Safety margin scaling contours}
    \end{figure}

\section{Distance Estimation Program}
\label{sec:distance_lp}

The goal of this paper is to develop a computationally tractable framework to compute the worst case (over all initial conditions) distance of closest approach to an unsafe set. Specifically, we aim to solve the following problem.
\begin{problem}[Distance \rev{Calculation}]\label{prob:problem1} Given semi-algebraic initial condition ($X_\rw{0})$ and unsafe ($X_u$) sets, solve the optimization problem \eqref{eq:dist_traj}.
\end{problem}
In many practical situations, it is sufficient to obtain interpretable lower bounds on the minimum distance. 
Thus, the following problem is also of interest.
\begin{problem}[Distance \rev{Estimation}]\label{prob:problem2} Given semi-algebraic initial condition set ($X_\rw{0}$), an unsafe ($X_u$) set, and a positive integer  $d$ (degree),  find a lower bound $p^*_d \leq P^*$ to the solution of optimization \eqref{eq:dist_traj}.
\end{problem}
As we will show in this paper (and under mild \rev{compactness and regularity} conditions), a convergent sequence of lower bounds $\{p^*_d\}$ that rise to  $ \lim_{d \rightarrow \infty} p_d^* = P^*$ may be constructed where each bound
$p^*_d$ is obtained by solving a finite dimensional \rw{\ac{SDP}}.

An optimizing trajectory of the Distance program \eqref{eq:dist_traj} may be described by a tuple $\rev{\ts^*} = (y^* \ x_0^*, \ t_p^*)$ \rw{using} 
Table \ref{tab:opt_traj_dist}.
    \begin{table}[ht]
        \centering
                \caption{\label{tab:opt_traj_dist}Characterization of optimal trajectory in distance estimation }
            \begin{tabular}{c l }
         $y^*$ &location on unsafe set of closest approach \\
         $x_0^*$ &initial condition to produce \rw{closest approach} \\
         $t^*_p$ &time to reach \rw{closest approach} from $x_0^*$ \\
        \end{tabular}
    \end{table}
        
        The relationship between these quantities for an optimal trajectory of \eqref{eq:dist_traj} is:
        \begin{equation}
            P^* = \rw{c(x(t^*_p \mid x_0^*); X_u)} = c(x(t^*_p \mid x_0^*), y^*).
        \end{equation}
        
    Figure \ref{fig:distance_example_reprise} plots the trajectory of closest approach to $X_u$ in dark blue. This minimal $L_2$ distance is $0.2831$, and the red curve is the level set of all points with a point-set distance $0.2831$ to $X_u$. 
    On the optimal trajectory, the blue circle is $x_0^* \approx (1.489, -0.3998)$, the blue star is $x_p^*\rw{=x(t^*\mid x_0)} \approx (0, -0.2997)$, and the blue square is $y^* \approx (-0.2002, -0.4998)$. The closest approach of $0.2831$ occurred at time $t^* \approx 0.6180$. Figure \ref{fig:safety_distance} plots the distance and safety margin contours for the set $X_u$. \rw{These distance contours for a given metric $c$ are independent of the way that $X_u$ is defined (within the same coordinate system).}
       \begin{figure}[h]
        \centering
        \includegraphics[width=0.5\linewidth]{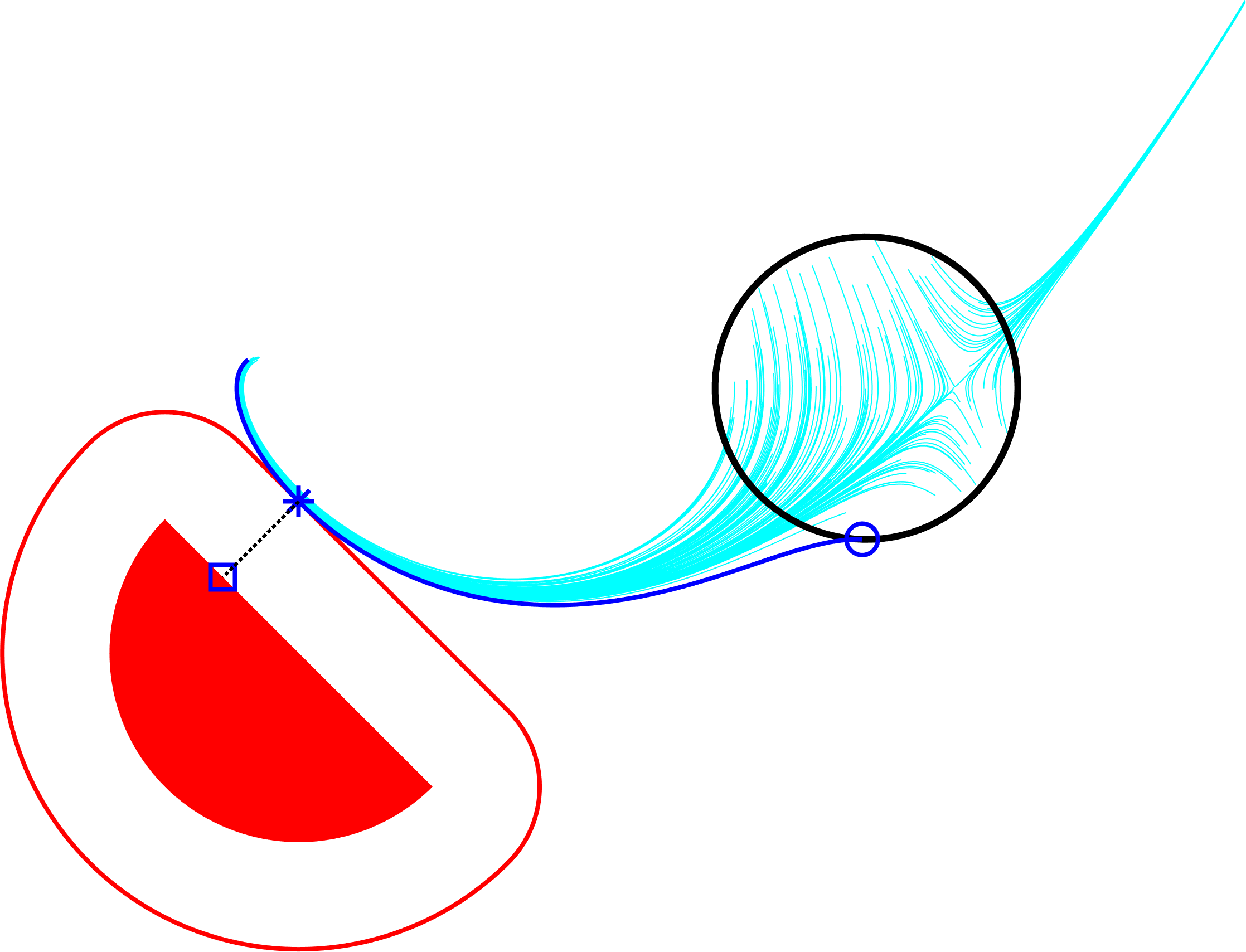}
        \caption{\label{fig:distance_example_reprise}$L_2$ bound of $0.2831$}
    \end{figure}

    \begin{figure}[ht]
        \centering
        \includegraphics[width=0.9\linewidth]{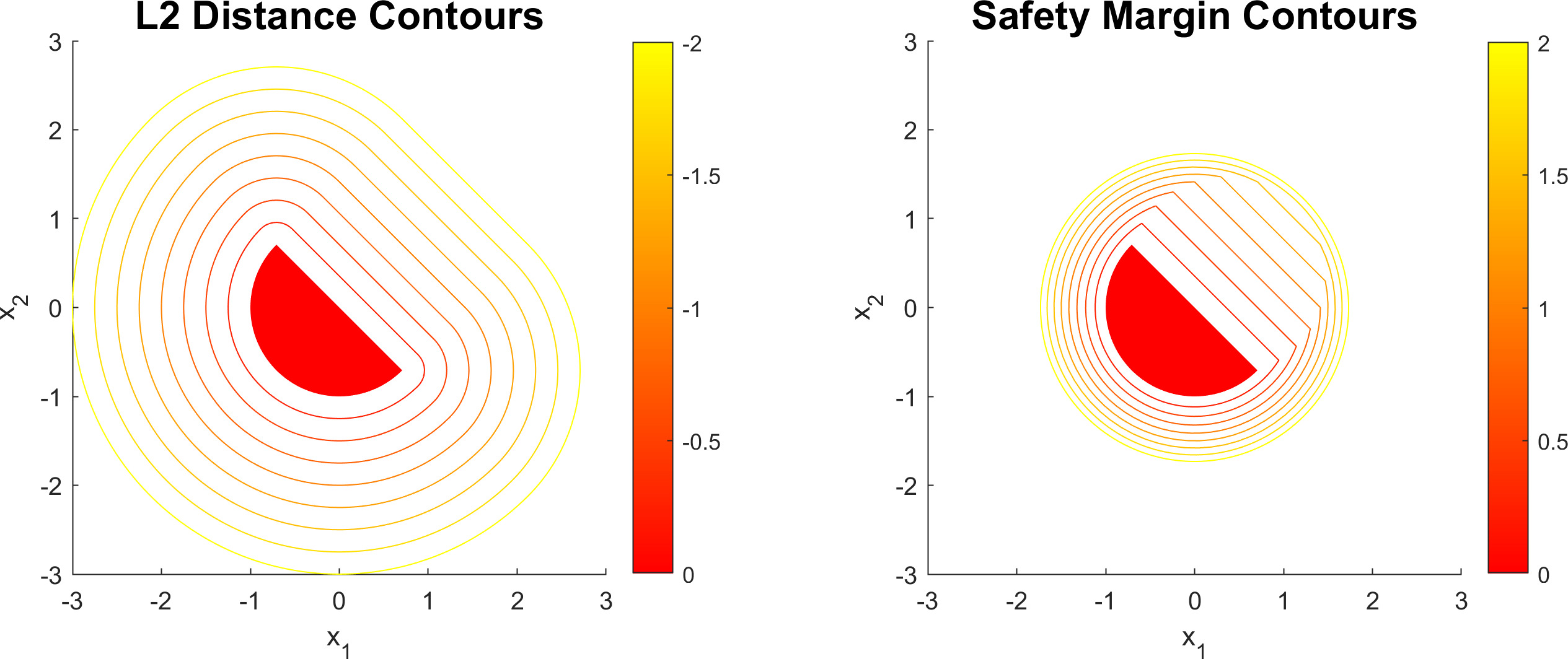}
        \caption{\label{fig:safety_distance} Comparison between $L_2$ distance and safety margin contours}
    \end{figure}

 \subsection{Assumptions}
The following assumptions are made in \rev{P}rogram \eqref{eq:dist_traj}:
\begin{enumerate}
\item[A1] The sets $[0, T], \ X , \ X_u\rw{, \ X_0}$ are \rw{all} compact, \rw{$X_0 \subset X$, and $n$ is finite.}.
    \item[A2] The function $f(t, x)$ is Lipschitz in each argument \rev{in the compact set $[0, T] \times X$}.
    \item[A3] The cost $c(x, y)$ is $C^0$ \rev{in $X\times X_u$.}
    \item[A4] \rev{If $x(t \mid x_0) \, \rw{\in \partial X}$ for some $t \in [0, T], \ x_0 \in X_0$, then $x(t' \mid x_0) \not\in X \ \forall t' \in \rw{(}t, T].$
    }
\end{enumerate}
 
\rev{
 A3 relaxes the requirement that $c$ should be  a metric, allowing for costs such as $\norm{x-y}_2^2$ in addition to the metric $\norm{x-y}_2$. The combination of A1 and A3 enforce that $c(x,y)$ is bounded inside $X \times X_u$ by the Weierstrass extreme value theorem.
Assumption A4 requires that trajectories do not return to \rw{$X$ after contacting the boundary $\partial X$}.}
\begin{remark}
\rw{An alternate (inequivalent and not contained) assumption that can be used for convergent distance estimation instead of A4 is that every trajectory starting from $X_0$ stays in $X$ for times $t \in [0, T]$ (employed in \cite{fantuzzi2020bounding}). Letting $\{X^\epsilon\}_{\epsilon \geq 0}$ be a family of compact (A1) sets under the relations $X \subset X^\epsilon \subset X$ and $X^{\epsilon'} \supset X^\epsilon$ for all $0 < \epsilon < \epsilon' < \infty$, the assumption of non-return (A4) for each $X^\epsilon$ strictly contains the assumption $X$ is an invariant set for all trajectories starting from $X_0$.} 
\end{remark}

 
\subsection{Measure Program}
\rev{The problem of $c^* = \min_{(x, y) \in X \times X_u} c(x,y)$ is identical to $c^* = \min_{(x, y) \in X \times X_u} \inp{c(x,y)}{\delta_{x} \otimes \delta_y}$ for Dirac measures $\delta_{x} \otimes \delta_y$. The Dirac restriction may be relaxed to minimization over the set of probability measures $c^* = \inp{c(x, y)}{\eta}, \eta \in \Mp{X \times X_u}, \inp{1}{\eta}= 1$ with no change in the objective value $c^*$.}
An infinite-dimensional \rw{convex} \ac{LP} in measures $(\mu_0, \mu_p, \mu, \eta)$ to \rev{bound from below} the distance closest approach to $X_u$ starting from $X_0$ \rw{may be developed}.

\begin{theorem}
\label{thm:meas_lower}
\rw{Suppose that $f \in C^0$ and A3 holds. Further impose that if $X_0 \subset X$ are both compact then A4 holds. Under these conditions, a lower bound for $P^*$ is,}
\begin{subequations} 
\label{eq:dist_meas}
    \begin{align}
    p^* &= \ \rev{\inf} \quad \inp{c(x, y)}{\eta} \label{eq:dist_meas_obj}\\
    &\pi^{x}_\# \eta = \pi^{x}_\# \mu_p \label{eq:dist_meas_marg}\\
    &\mu_p = \delta_0 \otimes\mu_0 + \Lie_f^\dagger \mu   \label{eq:dist_meas_liou}\\
    &\inp{1}{\mu_0} = 1 \label{eq:dist_meas_prob}\\
    & \mu_0 \in \Mp{X_0}, \  \eta \in \Mp{X \times X_u} \label{eq:dist_meas_joint} \\
    & \mu_p, \ \mu \in \Mp{[0, T] \times X}&\label{eq:dist_meas_occ}
\end{align}
\end{subequations}
\end{theorem}

\begin{proof}
\rev{Let $\ts = (y, x_0, t_p) \in X_u \times X_0 \times [0, T]$ be a tuple representing a trajectory with $x_p = x(t_p \mid x_0)$ achieving a distance $P = c(x_p, y)$. A set of measures \eqref{eq:dist_meas_joint}-\eqref{eq:dist_meas_occ} satisfying constraints \eqref{eq:dist_meas_marg}-\eqref{eq:dist_meas_occ} may be constructed from the tuple $\ts$.} The initial measure $\mu_0 =  \delta_{x=x_0}$, the peak (free-time terminal) measure $\mu_p = \delta_{t=t_p} \otimes \delta_{x=x_p}$ \rw{with $x_p= x(t_p\mid x_0)$}, and the joint measure $\eta = \delta_{x_p} \otimes \delta_{y=y}$, are all rank-one atomic probability measures. The measure \rev{$\mu$} is the \rw{occupation measure of $t \mapsto (t, x(t \mid x_0))$ in times $[0, t_p]$.  
} The distance objective \eqref{eq:dist_meas_obj} \rev{for the tuple $\ts$} may be evaluated as\rev{,}
\begin{equation}
    \inp{c(x, y)}{\eta} = \inp{c(x, y)}{\delta_{x=x_p} \otimes \delta_{y=y}} = c(x_p, y) = P. 
\end{equation}
\rw{The feasible set of \eqref{eq:dist_meas_marg}-\eqref{eq:dist_meas_occ} contains all measures constructed from trajectories by the above process, which immediately implies that $p^* \leq P^*$.}
\end{proof}

\begin{remark}
As a reminder, the term $\pi^{x}_\#$ from constraint \eqref{eq:dist_meas_marg} is the operator performing $x$-marginalization. Constraint \eqref{eq:dist_meas_marg} ensures that the $x$-marginals of $\eta$ and $\mu_p$ are equal\rev{: $\forall w \in C(X): \inp{w(x)}{\eta(x, y)} = \inp{w(x)}{\mu_p(t,x)}$}.
\end{remark}




\subsection{Function Program}

Dual variables $v(t, x) \in C^1([0, T] \times X)$, $ w(x) \in C(X)$, $\gamma \in \R$ over constraints \eqref{eq:dist_meas_marg}-\eqref{eq:dist_meas_prob} must be introduced to derive the dual LP to \eqref{eq:dist_meas}.
The Lagrangian $\scL$ of problem \eqref{eq:dist_meas} is:
    \begin{align}
    \scL &=   \inp{c(x,y)}{\eta} + \inp{v(t, x) }{\delta_0 \otimes\mu_0 + \Lie_f^\dagger \mu -\mu_p }     \label{eq:lagrangian_distance}\\
    &+ \inp{w(x)}{\pi^{x}_\# \mu_p - \pi^{x}_\# \eta} +  \gamma(1 - \inp{1}{\mu_0}).\nonumber
\end{align}
\rev{Recalling that $\forall \eta \in \Mp{X \times Y}$, $\rw{w} \in C(X)$ the relation that $\inp{\rw{w}(x)}{\eta(x,y)} = \inp{\rw{w}(x)}{\pi_\#^x \eta(x)}$ holds,} the Lagrangian $\scL$ in \eqref{eq:lagrangian_distance} may be reformulated as,
\begin{align}
    \scL&=\gamma + \inp{v(0, x) - \gamma}{\mu_0} + \inp{c(x, y)-w(x)}{\eta} \\
    &+\inp{w(x) - v(t, x)}{\mu_p} + \inp{\Lie_f v(t, x)}{\mu}. \nonumber
    \end{align}


The dual of program \eqref{eq:dist_meas} is provided by,
\begin{subequations}
\label{eq:dist_cont}
\begin{align}
    d^* = & \rev{\sup}_{\gamma, v, w} \ \rev{\inf}_{\mu_0, \mu_p, \mu, \eta} \scL \\
    = & \rev{\sup}_{\gamma \in \R} \quad \gamma & \\
    & v(0, x) \geq \gamma&  &\forall x \in X_0 \label{eq:dist_cont_init}\\
    & c(x, y)\geq w(x) & & \forall (x, y) \in X \times X_u \label{eq:dist_cont_cw}\\
    & w(x) \geq v(t, x) & &\forall (t, x) \in [0, T] \times X \label{eq:dist_cont_wv}  \\
    & \Lie_f v(t, x) \geq 0 & &\forall (t, x) \in [0, T] \times X \label{eq:dist_cont_f}\\
    & w \in C(X) \label{eq:dist_cont_w}\\
    &v \in C^1([0, T] \times X). & \label{eq:dist_cont_poly}
\end{align}
\end{subequations}

\begin{theorem}
\label{thm:strong_duality_dist}
Strong duality \rw{with $p^*=d^*$ and} \rev{attainment of optima} \rw{occurs} under assumptions A1-\rev{A4}.
\end{theorem}
\begin{proof}
\rev{See} Appendix \ref{app:duality}.


\end{proof}

\begin{remark}
\label{rmk:chain_lower}
The continuous function $w(x)$ is a lower bound on the point set distance $c(x; X_u)$ by constraint \eqref{eq:dist_cont_cw}. The auxiliary function $v(t, x)$ is in turn a lower bound on $w(x)$ by constraint \eqref{eq:dist_cont_wv}. This establishes a chain of lower bounds $v(t, x) \leq w(x) \leq c(x; X_u)$ holding $\forall (t, x) \in [0, T] \times X$.
\end{remark}

\begin{theorem}
\label{thm:relaxation_gap}
Under assumptions A1-\rev{A4}, \rev{$d^* = P^*$}.
\end{theorem}
\begin{proof}
\rev{This proof will show that for every arbitrary $\delta > 0$, there a exists a feasible $(\gamma, v, w)$ such that $P^* - \delta \leq d^* \leq P^*$. It therefore follows that $d^* = P^*$ under A1-A4.}

The relation $d^* = p^*$ holds by strong duality from Theorem \ref{thm:strong_duality_dist}, and the bound $p^* \leq P^*$ is valid by measure construction in Theorem \ref{thm:meas_lower}. Therefore, the bound from above $d^* \leq P^*$ is valid.

\rev{The function $c(x; X_u)$ is $C^0$, so an admissible choice of $w(x)$ is $w(x) = c(x; X_u)$.}
Appendix D of \cite{fantuzzi2020bounding} provides a proof that \rev{ (minimizing) peak estimation with $C^0$ state cost $w(x) = c(x; X_u)$} may be approximated to arbitrary accuracy by a $C^1$ auxiliary function $v(t, x)$\rev{. Such a $v$ may be constructed by finding a function $W \in C^1([0, T] \times X)$ satisfying (a modification of equations D.2 and D.3 of \cite{fantuzzi2020bounding} to account for $\inf$ rather than $\sup$),
\begin{subequations}
\label{eq:fantuzzi_W}
\begin{align}
    &\Lie_f W(t, x) \geq -\delta/(5T) & & \forall (t, x) \in [0, T] \times X \\
    &w(x) \geq W(t,x) - (2/5) \delta & & \forall (t, x) \in [0, T] \times X \\
    & W(0,x) \geq \gamma & & \forall x \in X_0 \\
    & \gamma \geq P^* - (2/5)\delta,
    \end{align}
\end{subequations}
with $v$ found from $W$ by,
\begin{equation}
\label{eq:fantuzzi_v}
    v(t,x) = W(t,x) - (2/5)\delta- \delta/(5T) (T-t).
\end{equation}

The function $W(t,x)$ may be constructed from a flow map for trajectories of $f$ by following the steps of Lemma D.2 of \cite{fantuzzi2020bounding}. 
\rw{We note that the procedure in \cite{fantuzzi2020bounding} assumes that trajectories starting in $X_0$ remain in $X$ for all $t \in [0, T]$ (A.1 in \cite{fantuzzi2020bounding}). By utilizing the non-return assumption A4 modifying the constructed auxiliary function in D.21 of \cite{fantuzzi2020bounding} such that $t_2$ (terminal time) occurs at the minimum of $t_2=T$ or when the trajectory touches $\partial X$ for the first time, we can utilize Appendix D of \cite{fantuzzi2020bounding} in constructing $W$ for this proof and extend \cite{fantuzzi2020bounding}'s applicability to the non-return case}.
Feasible $(\gamma, v, w)$ may be therefore found such that the bounds $P^* - \delta \leq d^* \leq P^*$ hold for every $\delta > 0$, which completes the proof.}
\end{proof}

\rev{
\begin{cor}
\label{cor:smooth}
\rw{For every $\delta > 0$, there exists a $\gamma \in \R$ and smooth functions $w\in C^\infty(X), \ v \in C^\infty([0, T] \times X)$ such that $P^* - \delta \leq \gamma \leq P^*$ in \eqref{eq:dist_cont}. Additionally, the $w$ and $v$ functions can be taken to be polynomials.}
\end{cor}
\begin{proof}
For every $\epsilon > 0$, there exists a Stone-Weierstrass approximation $\bar{w}(x) \in \R[x]$ over the compact set $X$ to the $C^0$ continuous function $c(x; X_u)$ such that \rw{$\norm{(c(x; X_u)-\epsilon) - \bar{w}(x)}_{C^0(X)} \leq \epsilon$}
. The function $\rw{\bar{w}}(x)$ is therefore a lower bound for $c(x; X_u)$ and satisfies constraints \rw{\eqref{eq:dist_cont_cw}} and \eqref{eq:dist_cont_w}. A $C^1$ function $v(t,x)$ may be constructed from \eqref{eq:fantuzzi_W} and \eqref{eq:fantuzzi_v} (following process from Theorem \ref{thm:relaxation_gap}) at a tolerance of \rw{$\delta' \in  (0, (2/5)\delta)$}. This auxiliary function $v$ may in turn be approximated \rw{using Theorem 1.1.2 of \cite{llavona1986approximation} by a polynomial} $G \in \R[t,x]$ such that \rw{$\norm{G - v}_{C^1} < \epsilon$. 
By following the process from the proof of Theorem 4.1 of \cite{fantuzzi2020bounding} with $i$-coordinate dynamics $f_i(t, x)$, we can choose $\epsilon$ such that
\begin{equation*}
    \epsilon < \delta'/\left(\max\left[2, 2T, 2T \max_i \norm{f_i(t, x)}_{[0, T] \times X}  \right]\right).
\end{equation*}

The collection of the polynomial auxiliary function $\bar{v}(t, x) = G(t, x) - \delta(1 - t/(2T))$  with the polynomial $\bar{w}$ and scalar $\gamma = P^* - \delta$ together satisfies inequalities \eqref{eq:dist_cont_init}-\eqref{eq:dist_cont_poly} strictly for each given $\delta$, thus proving the theorem.
}
\end{proof}
}

\section{Finite-Dimensional Programs}
\label{sec:distance_finite}
This section presents finite-dimensional \ac{SDP} truncations to the inifinite-dimensional \acp{LP} \eqref{eq:dist_meas} and \eqref{eq:dist_cont}.

\rev{
\subsection{Approximation Preliminaries}
We introduce notation and concepts about moments and \ac{SOS} polynomials that will be used in subsequent finite-dimensional programs. Refer to Appendix \ref{sec:moment_sos} for further detail (e.g. Archimedean structure, moment-\ac{SOS} hierarchy, conditions of convergence).
A basic semialgebraic set $\K = \{x \mid g_i(x) \geq 0, \ i = 1, \ldots, N_c\}$ is a set formed by a finite set of bounded-degree polynomial constraints. The $\alpha$-moment of a measure $\mu$ is $\bm_\alpha = \inp{x^\alpha}{\mu}$. \rw{Assuming that each constraint polynomial $g_k(x)$ has a representation as $g_k(x) = \sum_{\sigma \in \N^n} g_{k \sigma} x^\sigma$, then } the matrix $\M_d(\K \bm)$ formed by a moment sequence $\bm$ is the block-diagonal matrix formed by $\diag{[\bm_{\alpha +\theta}]_{\alpha, \theta \in \N^n_{\leq d}}, \forall k: [\sum_{\sigma \in \N^d} g_{k \sigma} \bm_{\alpha+\theta+\sigma}]_{\alpha, \theta \in \N^n_{\leq d - \deg{g_k}/2}}}$. 

A polynomial $p(x)$ is \ac{SOS} $(p(x) \in \Sigma[x])$ if there exists a finite integer $s$, a polynomial vector $v(x) \in \R[x]^s$, and a \ac{PSD} matrix $Q \in \psd_+^s$, such that $p(x) = v(x)^T Q v(x)$. \ac{SOS} polynomials are nonnegative over $\R^n$. A polynomial is \ac{WSOS} over a set $\K$ (expressed as $p(x) \in \Sigma[\K]$ if there exists $\forall k=0..N_c: \sigma_k \in \Sigma[x]$ such that $p(x) = \sigma_0(x) + \sum_k g_k(x) \sigma_k(x)$.

}
\subsection{LMI Approximation}
In the case where $c(x, y)$ and $f(t,x)$ are polynomial, \eqref{eq:dist_meas} may be approximated with a converging hierarchy of \acp{SDP}. Assume that that $X_0$, $X$,  and $X_u$ are Archimedean basic semialgebraic sets, each defined by a finite number of bounded-degree polynomial inequality constraints $X_0 = \{\rw{x \mid \ } g^0_k(x) \geq 0\}_{k=1}^{N_0}$, $X = \{\rw{x \mid\ } g^X_k(x) \geq 0\}_{k=1}^{N_X}
$, and $X_u = \{\rw{x \mid \ }g^U_{k} (x) \geq 0\}_{k=1}^{N_U}$.

The polynomial inequality constraints for $X_0, X, X_u$ are of degrees $d^0_k, d_k, d_k^U$ respectively.
The Liouville equation in \eqref{eq:dist_meas_liou} enforces a countably infinite set of linear constraints indexed by all possible $\alpha \in \N^n, \ \beta \in \N$,
\begin{align}
\label{eq:liou_mom}
    \inp{x^\alpha}{\mu_0} \delta_{\beta 0} + \inp{\Lie_f( x^\alpha t^\beta)}{\mu} - \inp{x^\alpha t^\beta}{\mu_p} &= 0.
\end{align}
The expression $\delta_{\beta 0}$ is the Kronecker Delta taking a value $\delta_{\beta 0} = 1$ when $\beta = 0$ and $\delta_{\beta 0}=0$ when $\beta \neq 0$. Let $(\rbm^0, \rbm^p, \rbm, \rbm^{\eta})$ be moment sequences for the measures $(\mu_0, \mu_p, \mu, \eta)$. Define $\textrm{Liou}_{\alpha \beta}(\rbm^0, \rbm, \rbm^p)$ as the linear relation induced by  \eqref{eq:liou_mom} at the test function $x^\alpha t^\beta$ in terms of moment sequences. The polynomial metric $c(x, y)$ may be expressed as $\sum_{\alpha, \gamma} c_{\alpha \gamma} x^\alpha y^\gamma$ for multi-indices $\alpha, \gamma \in \N^n$. The complexity of dynamics $f$ induces a degree $\Tilde{d}$ as $\Tilde{d} = d + \ceil{\textrm{deg}(f)/2} - 1$.
The degree-$d$ \ac{LMI} relaxation of \eqref{eq:dist_meas} with moment sequence variables $(\rbm^0, \rbm^p, \rbm, \rbm^\eta)$ is
\begin{subequations}
\label{eq:dist_lmi}
\begin{align}
    p^*_d = & \min \quad \textstyle\sum_{\alpha,\gamma} c_{\alpha \gamma} \rbm_{\alpha \gamma}^\eta. \label{eq:dist_lmi_obj} \\
    &\rbm^\eta_{\alpha 0} = \rbm^p_{\alpha 0} & &\forall \alpha \in \N^{n}_{\leq 2d} \label{eq:dist_lmi_marg}\\
    &\textrm{Liou}_{\alpha \beta}(\rbm^0, \rbm^p, \rbm) = 0 & &\forall (\alpha, \beta) \in \N^{n+1}_{\leq 2d} \label{eq:dist_lmi_flow}\\
    & \rbm^0_0 = 1 \label{eq:dist_lmi_prob} \\
    & \M_d(\rev{X_0} \rbm^0)  \succeq 0 \\
    &\M_d(\rev{([0, T] \times X)}\rbm^p) \succeq 0 \\
    &\M_{\Tilde{d}}(\rev{([0, T] \times X)}\rbm) \succeq 0 \\
    &\M_d(\rev{(X \times X_u)}\rbm^\eta) \succeq 0 \label{eq:dist_lmi_joint}.
\end{align}
\end{subequations}
Constraints \eqref{eq:dist_lmi_marg}-\eqref{eq:dist_lmi_prob} are finite-dimensional versions of constraints \eqref{eq:dist_meas_marg}-\eqref{eq:dist_meas_prob} from the measure LP. \rw{In order to ensure convergence $\lim_{d \rightarrow \infty} p^*_d = p^*$ we must establish that all moments of measures are bounded.}

\begin{lem}
\label{lem:bounded}
\rev{The masses of all measures in \eqref{eq:dist_meas} are finite (\rw{uniformly} bounded) if A1-\rev{A4} hold.}
\end{lem}
\begin{proof}
Constraint \eqref{eq:dist_meas_prob} imposes that $\inp{1}{\mu_0} = 1$, which further requires that $\inp{1}{\mu_p} = \inp{1}{\mu_0} = 1$ by constraint \eqref{eq:dist_meas_liou} ($v(t, x) = 1$) and $\inp{1}{\mu_p} = \inp{1}{\eta} = 1$ ($w(x) = 1$). The occupation measure $\mu$ likewise has bounded mass with $\inp{1}{\mu} = \inp{t}{\mu^p} < T$ by constraint \eqref{eq:dist_meas_liou} ($v(t, x) = t$).
\end{proof}

\rev{
\begin{lem}
The measures$(\mu_0, \mu_p, \mu, \eta)$ all have finite moments under Assumptions A1-A\rev{4}.
\label{lem:finite_mom}
\end{lem}
\begin{proof}
A sufficient condition for 
 a  measure $\tau \in \Mp{X}$ with compact support   to be bounded is to have finite mass $\inp{1}{\tau}$. In our case, the support of all measures $(\mu_0, \mu_p, \mu, \eta)$ 
 are compact sets \rw{by A1}.
 Further, under Assumptions A1-A4, all of these measures have bounded mass (Lemma \ref{lem:bounded}).
 This sufficiency is satisfied by all measures $(\mu_0, \mu_p, \mu, \eta)$.
\end{proof}
}
\begin{theorem}
\label{thm:lmi_convergence}
When $T$ is finite and $X_0,  X, X_u$ are all \rev{Archimedean}, the sequence of lower bounds $p^*_{d} \leq p^*_{d+1} \leq p^*_{d+2} \ldots$ will approach  $p^*$ as $d$ tends towards $\infty$.
\end{theorem}
\begin{proof}
This convergence is assured \rev{by Corollary 8 of \cite{tacchi2022convergence} under the Archimedean assumption and Lemma \ref{lem:bounded}.}
\end{proof}



\begin{remark}
\rw{Non-polynomial} $C^0$ cost functions $c(x, y)$ may be approximated by polynomials $\tilde{c}(x, y)$ through the Stone-Weierstrass theorem in the compact set $X \times Y$. For every $\epsilon > 0$, there exists a $\tilde{c}(x, y) \in \R[x, y]$ such that $\max_{x\in X, y \in X_u} \abs{c(x, y) - \tilde{c}(x, y)} \leq \epsilon$. Solving the peak estimation problem \eqref{eq:dist_meas} with cost $\tilde{c}(x, y)$ as $\epsilon \rightarrow 0$ will yield convergent bounds to $P^*$ with cost $c(x, y)$. Section \ref{sec:polyhedral} offers an alternative peak estimation problem \rw{using polyhedral lifts} for costs \rw{comprised by the maximum of a set of functions}.
\end{remark}

\subsection{Numerical Considerations}

A moment matrix with $n$ variables in degree $d$ has dimension $\binom{n+d}{d}$. \rev{The sizes of moment matrices associated with a $d$ relaxation of  Problem \eqref{eq:dist_lmi} with state  $x \in \mathbb{R}^n$, dynamics $f(t, x)$, and induced dynamic degree $\tilde{d}$, are listed in Table \ref{tab:mom_size}.}


\begin{table}[h]
\caption{\label{tab:mom_size}Sizes of moment matrices in \ac{LMI} \eqref{eq:dist_lmi}}
        \centering
        \begin{tabular}{c c c c c}
             Moment&  $\M_d(\rbm^0)$ & $\M_d(\rbm^p)$ & $\M_{\tilde{d}}(\rbm)$ & $\M_d(\rbm^\eta)$ \\
             \\
             Size & $\binom{n+d}{d}$ & $\binom{1+n+d}{d}$ & $\binom{1 + n+\tilde{d}}{\tilde{d}}$ & $\binom{2n+d}{d}$         \end{tabular}
    \end{table}

The computational complexity \rev{of} solving the \rw{SDP formulation of } \ac{LMI} \eqref{eq:dist_lmi} scales polynomially as the largest matrix size in Table \ref{tab:mom_size}, usually $\M_d(\rbm^\eta)$, except in cases where $f(t, x)$ has a high polynomial degree. 

\begin{remark}
\label{rmk:combine_mup_eta}
The measures $\mu_p$ and $\eta$ may in principle be combined \rev{into} a larger measure $\tilde{\eta} \in \Mp{[0, T] \times X \times X_u}$. The Liouville equation \eqref{eq:dist_meas_liou} would then read $\pi^{tx}_\# \tilde{\eta} = \delta_{0}\otimes \mu_0 + \Lie_f^\dagger \mu $, and a valid selection of $\tilde{\eta}$ given an optimal trajectory is $\tilde{\eta} = \delta_{t=t_p^*}\otimes\delta_{x=x_p^*}\otimes\delta_{y=y^*}$ \rw{with $x_p^*=x(t_p^* \mid x_0^*)$}. The measure $\tilde{\eta}$ is defined over $2n+1$ variables, and the size of its moment matrix at a degree $d$ relaxation is $\binom{1+2n+d}{d}$, as compared to $\binom{2n+d}{d}$ for $\eta$. We elected to split up the measures as $\mu_p$ and $\eta$ to reduce the number of variables in the largest measure, and to ensure that the objective \eqref{eq:dist_meas_obj} is interpretable as an earth-mover distance \rev{(from optimal transport literature\cite{villani2008optimal})} between $\pi^x_\# \mu_p$ and a probability distribution over $X_u$ (absorbed into $\pi^x_\# \eta$).
\end{remark}


\begin{remark}
The distance problem \eqref{eq:dist_traj} may also be treated as a peak estimation problem \eqref{eq:peak_traj_std} with cost $p(x,y) = \rev{-}c(x, y)$, initial set $X_0 \times X_u$, \rev{$x$-dynamics} $\dot{x}(t) = f(t, x(t))$, \rev{and $y$-dynamics} $\dot{y}(t) = 0$.
The moment matrix $\M_d[\rbm]$ associated with this peak estimation problem's occupation measure (\ac{LMI} relaxation of program \eqref{eq:peak_meas}) would have size $\binom{1+2n+\tilde{d}}{d}$. 
\rev{T}his size is greater than any of the sizes written in Table \ref{tab:mom_size}.
\end{remark}

\begin{remark}
The atom-extraction-based recovery Algorithm 1 from \cite{miller2020recovery} may be used to approximate near-optimal trajectories if the moment matrices $\M_d(\rbm^0)$, $\M_d(\rbm^p)$, and $\M_d(\rbm^\eta)$ are each low rank. If these matrices are all rank-one, then the near-optimal points $(x_p, y, x_0,  t_p)$ may be read directly from the moment sequences $(\rbm^0, \rbm^p, \rbm^\eta)$. The near optimal points from Figure \ref{fig:flow_recovery} were recovered at the degree-4 relaxation of \eqref{eq:dist_lmi}. The top corner of the moment matrices $\M_d(\rbm^0)$, $\M_d(\rbm^p)$, and $\M_d(\rbm^\eta)$ (containing moments of orders 0-2) have second-largest eigenvalues of $1.87 \times 10^{-5}$, $8.82 \times 10^{-6}$, $5.87 \times 10^{-7}$ respectively, as compared to the largest eigenvalues of $3.377$, $1.472$, $1.380$. 
\end{remark}

\rev{
\subsection{SOS Approximation}

The degree-$d$ \ac{WSOS} truncation of program \eqref{eq:dist_cont} is,
\begin{subequations}
\label{eq:dist_sos}
\begin{align}
    d^*_d = \ &\max_{\gamma \in \R} \quad \gamma\\
    & v(0, x) - \gamma \in \Sigma[X_0]_{\leq 2d}\label{eq:dist_sos_x0}\\
    & c(x, y) - w(x) \in \Sigma[X \times X_u]_{\leq 2d}\label{eq:dist_sos_cw}\\
    & w(x) - v(t, x)\in \Sigma[[0, T] \times X]_{\leq 2d} \label{eq:dist_sos_wv}  \\
    & \Lie_f v(t, x)  \in \Sigma[[0, T] \times X]_{\leq 2\tilde{d}} \label{eq:dist_sos_f}\\
    & w \in \R[x]_{\leq 2d} \label{eq:dist_sos_w}\\
    &v \in \R[t,x]_{\leq 2d}. \label{eq:dist_sos_poly}
\end{align}
\end{subequations}

\begin{theorem}
Strong duality holds with $p^*_k = d^*_k$ for all $k \in \N$ between \eqref{eq:dist_lmi} and \eqref{eq:dist_sos} under assumptions A1-A5.
\end{theorem}
\begin{proof}
Refer to \rev{Corollary 8 of \cite{tacchi2022convergence} (Archimedean condition and bounded masses), as well as} to the proof of Theorem 4 and Lemma 4 in Appendix D of \cite{henrion2013convex}. 
\end{proof}
}
\section{Exploiting Correlative Sparsity}
\label{sec:sparsity}

Many costs $c(x, y)$ exhibit a\rev{n additively} separable structure, such that $c$ can be decomposed into the sum of new terms $c(x, y) = \sum_i c_i(x_i, y_i)$. Each term $c_i$ in the sum is a function purely of $(x_i, y_i)$. Examples include the $L_p$ family of distance functions, such as the squared $L_2$ cost $c(x, y) = \sum_i (x_i - y_i)^2$. The theory of Correlative Sparsity in polynomial optimization, briefly reviewed below, can be used to substantially reduce the computational complexity entailed in solving the distance estimation \rw{\acp{SDP}} when $c$ is \rev{additively} separable \cite{waki2006sums}. \rev{This decomposition does not require prior structure on the set $X \times X_u$.} Other types of reducible structure (if applicable) include Term Sparsity \cite{wang2021tssos}, symmetry \cite{riener2013exploiting}, and network dynamics \cite{schlosser2020sparse}. These forms of structure may be combined if present, such as in Correlative and Term Sparsity \cite{wang2021chordal}.

\subsection{Correlative Sparsity Background}
Let $\K = \{x \mid g_k(x) \geq 0, \ k = 1, \ldots, N\}$ be an Archimedean basic semialgebraic set and $\phi(x)$ be a polynomial.
The \rev{\ac{CSP}} associated to $(\phi(x), g)$ is a graph $\gs(\vs, \es)$ with vertices $\vs$ and edges $\es$. 
Each of the $n$ vertices in $\vs$ corresponds to a variable $x_1, \ldots, x_n$. 
An edge $(x_i, x_j) \in \es$ appears if variables $x_i$ and $x_j$ are multiplied together in a monomial in $\phi(x)$, or if they appear together in at least one constraint $g_k(x)$ \cite{waki2006sums}.

The correlative sparsity pattern of $(\phi(x), g)$ may be characterized by sets $I$ of variables and sets $J$ of constraints. The $p$ sets $I$ should satisfy the following two properties:

\begin{enumerate}
    \item (Coverage)  $\bigcup_{j=1}^p I_j = \vs$
    \item (Running Intersection Property) For all $k = 1, \ldots, p-1$:  $I_{k+1} \cap \bigcup_{j=1}^k I_j\subseteq I_s$ for some $s \leq k$ 
\end{enumerate}

Equivalently, the sets $I$ are the maximal cliques of a chordal extension of $\gs(\vs, \es)$ \cite{vandenberghe2015chordal}.
The sets $J = \{J_i\}_{i = 1}^{p}$ are a partition over constraints $g_k(x) \geq 0$. The number $k$ is in $J_i$ for $k = 1, \ldots N_X$ if all variables involved the constraint polynomial $g_k(x)$ are contained within the set $I_i$. Let the notation $x(I_i)$ denote the variables in $x$ that are members of the set $I_i$.
A \rev{sufficient} sparse representation of  positivity certificates may be developed for $(\phi(x), g)$ satisfying an admissible correlative sparsity pattern $(I, J)$  \cite{lasserre2006convergent}:
    \begin{align}
        & \phi(x) =\textstyle \sum_{i=1}^p\sigma_{i0}(x(I_i)) + \sum_{k\rev{\in}J_i} {\sigma_k(x(I_i))g_k(x)} \label{eq:putinar_sparse}\\
        & \sigma_{i0}(x) \in \Sigma[x(I_i)] \qquad \sigma_k(x) \in \Sigma[x(I_i)] \qquad \forall i = 1, \ldots, p. \nonumber
    \end{align}

Equation \eqref{eq:putinar_sparse} is a sparse version of the Putinar certificate in \eqref{eq:putinar}. \rev{The sparse \rw{certificate} \eqref{eq:putinar_sparse} is additionally necessary \rw{for the $\gs$-sparse polynomial $\phi(x)$ to be positive over $\mathbb{K}$}  if $(I, J)$ satisf\rw{ies} the Running Intersection Property and a sparse Archimedean property holds: that there exist finite constants $R_i > 0$ for $i=1..n$ such that $R_i^2 - \norm{x(I_i)}^2_2$ is in the quadratic module \eqref{eq:quad_module} of constraints $Q[\{g_k\}_{k \in J_i}]$ \cite{lasserre2006convergent}.}



\subsection{Correlative Sparsity for Distance Estimation}
\label{sec:unsafe_csp} 

Constraint \eqref{eq:dist_cont_cw} will exhibit correlative sparsity when $c(x, y)$ is \rev{additively} separable,
\begin{align}
\label{eq:csp_poly}
    \textstyle \sum_{i=1}^n c_i(x_i, y_i) - w(x) \geq 0 & & \forall (x, y) \in X \times X_u.
\end{align}

The product-structure support set of Equation \eqref{eq:csp_poly} may be expressed as,
\begin{align}
\label{eq:csp_support}
    X \times X_u = \{(x, y) \mid &g_1(x)\geq 0, \ldots g_{N_X}(x) \geq 0, \\
    &g_{N_X+1}(y) \geq 0, \ldots g_{N_X+N_U}(y) \geq 0\}.\nonumber 
\end{align}
    

\rev{The correlative sparsity graph of \eqref{eq:csp_poly} is the graph Cartesian product of the complete graph $K_n$ by $K_2$, and is visualized at $n=4$ by \rw{the nodes and black lines in}  Figure \ref{fig:CSP_extend}.}     
Black lines imply that there is a link between variables. The  black lines are drawn between each pair $(x_i, y_i)$ from the cost terms $c_i$. 
The polynomial $w(x)$ involves mixed monomials of all variables $(x)=(x_1,x_2,x_3,x_4)$. 
Prior knowledge on the constraints of $X_u$ are not assumed in advance, so the variables are $(y)=(y_1,y_2,y_3,y_4)$ joined together. 
\Iac{CSP} $(I, J)$ associated with this system is,
\begin{align*}
    I_1 &= \{x_1,x_2,x_3,x_4, y_1\} & J_1 &= \{1, \ldots, N_X\}\\
    I_2 &= \{x_2,x_3,x_4, y_1, y_2\} & J_2 &= \varnothing \\
    I_3 &= \{x_3,x_4, y_1, y_2, y_3\} & J_3 &= \varnothing \\
    I_4 &= \{x_4, y_1, y_2, y_3, y_4\} & J_4 &= \{N_X+1, \ldots, N_X+N_U\}. \\
\end{align*}
Figure \ref{fig:CSP_extend} illustrates a chordal extension of the \ac{CSP} graph with new edges displayed as red dashed lines. These new edges appear by connecting all variables in $I_1$ together in a clique, and by following a similar process for $I_2, \ldots I_4$.






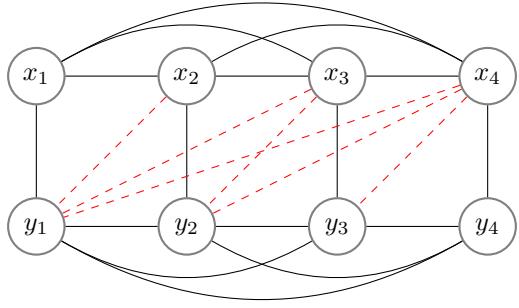
\begin{figure}[!htbp]
\centering
\begin{tikzpicture}[every node/.style={circle, draw=black!50, thick, minimum size=7.5mm}]
\node (n1) at (-3,4) {$x_1$};
\node (n2) at (-1,4) {$x_2$};
\node (n3) at (1,4) {$x_3$};
\node (n4) at (3,4) {$x_4$};
\node (n5) at (-3,2) {$y_1$};
\node (n6) at (-1, 2) {$y_2$};
\node (n7) at (1, 2) {$y_3$};
\node (n8) at (3,2) {$y_4$};

\draw (n1)--(n2);
\draw (n2)--(n3);
\draw (n3)--(n4);
\draw (n1) to [bend left] (n4);
\draw (n1) to [bend left] (n3);
\draw (n2) to [bend left] (n4);

\draw (n5)--(n6);
\draw (n6)--(n7);
\draw (n7)--(n8);
\draw (n5) to [bend right] (n8);
\draw (n6) to [bend right] (n8);
\draw (n5) to [bend right] (n7);

\draw (n1)--(n5);
\draw (n2)--(n6);
\draw (n3)--(n7);
\draw (n4)--(n8);

\draw[dashed, color=red] (n2)--(n5);
\draw[dashed, color=red] (n3)--(n5);
\draw[dashed, color=red] (n4)--(n5);
\draw[dashed, color=red] (n3)--(n6);
\draw[dashed, color=red] (n4)--(n6);
\draw[dashed, color=red] (n4)--(n7);

\end{tikzpicture}
\caption{\label{fig:CSP_extend} \ac{CSP} with 4-States and Chordal Extension}
\end{figure}

For a unsafe distance bounding problem with a \rev{additively} separable $c(x,y) = \sum_{i} c(x_i, y_i)$ with $n$ states, the correlative sparsity pattern $(I, J)$ is,
\begin{align}
    I_1 &= \{x_1,\ldots,x_n, y_1\} & J_1 &= \{1, \ldots, N_X\} \label{eq:csp_dist}\\
    I_i &= \{x_i,\ldots x_n, y_
    1, \ldots y_i\} & J_i &= \varnothing,   \qquad \forall i = 2, \ldots, n-1 \nonumber\\
    I_n &= \{x_n, y_1, \ldots, y_n\} & J_n &= \{N_X+1, \ldots, N_X+N_U\}.  \nonumber
\end{align}

A total of $(n-1)n/2$ new edges are added in the chordal extension.
\rev{Letting $y_{1:i}$ be the collection of variables $(y_1, y_2, \ldots, y_i)$ for an index $i \in 1..n$ (and with a similar definition for $x_{i:n}$), a} correlatively sparse certificate of positivity for constraint \eqref{eq:dist_cont_cw} is,
    \begin{align}
    \label{eq:csp_putinar}
        \sum_{i=1}^n c_i(x_i, y_i) - w(x)  &= \sum_{i=1}^n\sigma_{i0}(x_{i:n}, y_{1:i}) + \sum_{k=1}^{N_X} {\sigma_k(x, y_1)g_k(x)} \nonumber \\
        &+ \sum_{k=N_X + 1}^{N_X+N_U} {\sigma_k(x_n, y)g_k(y)},
        \end{align}
with sum-of-squares multipliers,
        
        \begin{align}
        & \sigma_{i0}(x, y) \in \Sigma[x_{i:n}, y_{1:i}] & &\nonumber 
        \forall i = 1, \ldots, p \\
        & \sigma_k(x, y) \in \Sigma[x, y_1] & &
         \forall k = 1, \ldots, N_X \\
        & \sigma_k(x, y) \in \Sigma[x_n, y] & &
         \forall k = N_X +1, \ldots, N_X+N_U.\nonumber 
    \end{align}

\rev{The application of correlative sparsity to the distance problem replaces constraint  \eqref{eq:dist_sos_cw} by \eqref{eq:csp_putinar}.}

\rev{
\begin{remark}
\label{rmk:csp_nonunique}
The CSP decomposition in 
\eqref{eq:csp_dist} is nonunique. As an example, the following decompositions are all valid for $n=3$ (satisfy Running Intersection Property),
\begin{align*}
    I_1 &= \{x_1,x_2,x_3,y_1\} & I_1' &= \{x_1,x_2,x_3,y_3\} \\
    I_2 &= \{x_2,x_3,y_1, y_3\} & I_2' &= \{x_1,x_2,y_2,y_3\} \\
    I_3 &= \{x_2,y_1, y_2, y_3\} & I_3' &= \{x_1,y_1,y_2,y_3\}
\end{align*}
\end{remark}
}

The original constraint $\eqref{eq:dist_cont_cw}$ is dual to the joint measure $\eta \in \Mp{X \times Y}$. Correlative sparsity may be applied to the measure program by splitting $\eta$ into new measures $\eta_1 \in \Mp{X \times \R}, \eta_n \in \Mp{\R \times X_u}$ and $\eta_i \in \Mp{\R^{n+1}}$ for $i = 2, \ldots, n-1$ following the procedure in \cite{lasserre2006convergent}. These measures will align on overlaps with $\pi^{I_i \cap I_{i+1}}_\# \eta_{i} = \pi^{I_i \cap I_{i+1}}_\# \eta_{i+1}, \ \forall i = 1, \ldots, n-1$. At a degree $d$ relaxation, the moment matrix of $\eta$ in \eqref{eq:dist_lmi} has size $\binom{2n + d}{d}$. Each of the $n$ moment matrices of $\{\eta_i\}^n_{i=1}$ has a size of $\binom{n+1+d}{d}$. For example, a problem with  $n=4, d=4$ will have a moment matrix for $\eta$ of size $\binom{12}{4} = 495$, while the moment matrices for each of the $\eta_{(1:4)}$ are of size $\binom{9}{4} = 126$. 
\section{Shape Safety}
\label{sec:shape}

 The distance estimation problem may be extended to sets or shapes travelling along trajectories, bounding the minimum distance between points on the shape and the unsafe set. 
An example application is in quantifying safety of rigid body dynamics, \rw{such as} finding the closest distance between all points on an airplane and points on a mountain. 

\subsection{Shape Safety Background}
 

Let $X \subset \R^n$ be a region of space with unsafe set $X_u$, and $c(x, y)$ be a distance function. The state $\omega \in \Omega$ (such as position and angular orientation) follows dynamics $\dot{\omega}(t) = f(t, \omega)$ between times $t \in [0, T]$. A trajectory is $\omega(t \mid \omega_0)$ for some initial state $\omega_0 \in \Omega_0 \rw{\subset \Omega}$. The shape of the object is a set $S$. There exists a mapping $A(s; \omega): S \times \Omega \rightarrow X$ that provides the transformation between local coordinates on the shape $(s)$ to global coordinates in $X$. 

Examples of a shape traveling along trajectories are detailed in Figure \ref{fig:shape_travel}. The shape $S = [-0.1, 0.1]^2$ is the pink square. The left hand plot is a pure translation \rev{after} a $5\pi/12$ radian rotation, and the right plot involves a rigid body transformation. 

\begin{figure}[ht]
    \centering
    \includegraphics[width=\linewidth]{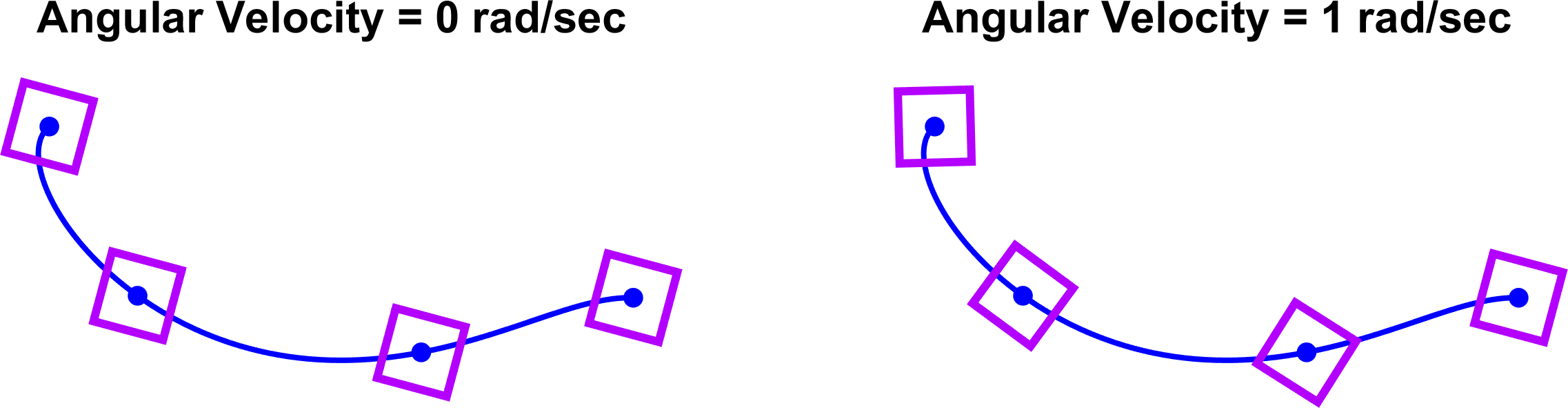}
    \caption{\label{fig:shape_travel}Shape moving and rotating along Flow \eqref{eq:flow} trajectories}
    
\end{figure}

The distance estimation task with shapes is to bound,
\begin{equation}
    \label{eq:shape_dist}
    \begin{aligned}
    P^* = & \rev{\inf}_{t,\:\omega_0 \in \Omega_0, \: s \in S, \: y \in X_u} c(A(s; \: \omega(t \mid \omega_0)), \: y) \\
    & \dot{\omega}(t) = f(t, \omega), \quad \forall t \in [0, T].
    \end{aligned}
\end{equation}

For each trajectory in state $\omega(t \mid \omega_0)$, problem \eqref{eq:shape_dist} ranges over all points in the shape $s \in S$ and points in the unsafe set $y \in X_u$ to find the closest approach.  An optimal trajectory of the shape distance program may be expressed as $\ts^*_s = (y^*, s^*, \omega_0^*, t^*_p)$ \rw{with $\omega^*_p = \omega(t_p^* \mid \omega_0^*),$ $x^*_p = A(s^*; \omega^*_p)$ and }
\[P^* = \rev{c(A(s^*; \omega^*_p), X_u) }= c(A(s^*; \omega(t_p^*\mid \omega^*_0)), y^*).\]

 \subsection{Assumptions}
The following assumptions are made in the Shape Distance program \eqref{eq:shape_dist}:
\begin{enumerate}
\item[A1'] The sets $[0, T], \  \Omega, \ S, \  X, \ X_u$ are compact \rw{and $\Omega_0 \subset \Omega$}.
    \item[A2'] The function $f(t, \omega)$ is Lipschitz in each argument.
    \item[A3'] The cost $c(x, y)$ is $C^0$.
     \item[\rw{A4'}] The coordinate transformation function $A(s; \omega)$ is $C^0$.
    \item[A5'] \rev{If $\omega(t \mid \omega_0) \rw{\in \partial}\Omega$ for some $t \in [0, T], \ \omega_0 \in \Omega_0$, then $\omega(t \mid \omega_0)\not\in \Omega \ \forall t' \in \rw{(}t, T].$}
    \item[A6'] \rev{If $\exists s \in S$ such that $A(s; \omega(t \mid \omega_0)) \not\in  X$ \rw{or $A(s; \omega(t \mid \omega_0)) \in \partial X$} for some $t \in [0, T], \ \omega_0 \in \Omega_0$, then $A(s; \omega(t' \mid \omega_0)) \not\in X \ \forall t' \in (t, T].$}
\end{enumerate}
\rw{An alternative assumption used instead of A5'-A6' is that $\omega(t\mid \Omega_0)$ stays in $\Omega$ for all $\omega_0 \in \Omega_0$ and $A(s; \omega(t \mid \omega_0)) \in X$ for all $s \in S, t \in [0, T]$.}

\subsection{Shape Distance Measure Program}

\rev{Program \eqref{eq:shape_dist} involves a distance objective $c(x, y)$, where the point $x = A(s; \omega)$ is given by a coordinate transformation between body coordinates $s$ and the evolving orientation $\omega$. In order to formulate a measure program to \eqref{eq:shape_dist}, a shape measure $\mu_s \in \Mp{S \times \Omega}$ may be added to bridge the gap between the changing orientation $\dot{\omega}$ and the comparison distance $x$. The shape measure contains information on the orientation $\omega$ and body coordinate $s$ that yields the closest point $x$,

\begin{subequations}
\begin{align}
\inp{z(\omega)}{\mu_p(t, \omega)} &= \inp{z(\omega)}{\mu_s(s, \omega)} & & \forall z \in C(\Omega) \label{eq:shape_demo_meas_marg_shape}\\
    \inp{w(x)}{\eta(x,y)} &= \inp{w(A(s;\omega))}{\mu_s(s, \omega)}  & & \forall w \in C(X). \label{eq:shape_demo_meas_marg_coord}
\end{align}
\end{subequations}

The shape measure $\mu_s$ chooses the worst-case  body coordinate $s$ and orientation $\omega$ from $\mu_p$ \eqref{eq:shape_demo_meas_marg_shape}, such that the point $x = A(s; \omega)$ comes as close as possible to the unsafe set's coordinate  $y$ \eqref{eq:shape_demo_meas_marg_coord}. We retain the coordinate $x$ in order to decrease the computational complexity of the \rw{\acp{SDP}}, as elaborated upon further in Remark \ref{rmk:combine_mup_eta}.
}

\rev{The} infinite dimensional measure program \rev{that lower bounds \eqref{eq:shape_dist}} is,
\begin{subequations}
\label{eq:shape_meas}
    \begin{align}
    p^* &= \ \rev{\inf} \quad \inp{c(x, y)}{\eta} \label{eq:shape_meas_cost}\\
    &\mu_p = \delta_0 \otimes\mu_0 + \Lie_f^\dagger \mu \label{eq:shape_meas_liou}  \\
    &\pi^{\omega}_\# \mu_p = \pi^{\omega}_\# \mu_s \label{eq:shape_meas_marg_shape}\\
    &\pi^{x}_\# \eta = A(s;\omega)_\# \mu_s \label{eq:shape_meas_marg_coord}\\
    &\inp{1}{\mu_0} = 1 \label{eq:shape_meas_marg_prob}\\
    &  \mu_0 \in \Mp{\Omega_0}, \ \eta \in \Mp{X \times X_u} \\
    & \mu_s \in \Mp{\Omega \times S}  \\
    & \mu_p, \ \mu \in \Mp{[0, T] \times \Omega}.
\end{align}
\end{subequations}
Constraint \eqref{eq:dist_meas_marg} in the original distance formulation is now split between \eqref{eq:shape_meas_marg_shape} and \eqref{eq:shape_meas_marg_coord} \rev{(which are equivalent to \eqref{eq:shape_demo_meas_marg_coord} and \eqref{eq:shape_demo_meas_marg_shape})}. 
Problem \eqref{eq:shape_meas} inherits all convergence and duality properties of the original \eqref{eq:dist_meas} under the appropriately modified set of assumptions A1'-A6'.

\rev{
\begin{theorem}
\label{thm:shape_meas_lower}
\rw{Under A3'-A4' (and additionally A5'-A6' when all sets in A1' are compact possibly excluding $[0, T]$), }
the Shape programs \eqref{eq:shape_dist} and \eqref{eq:shape_meas} are related by $p^* \leq P^*$.
\end{theorem}
\begin{proof}
This proof will follow the same pattern as Theorem \ref{thm:meas_lower}'s proof. A set of measures that are feasible solutions for the constraints of \eqref{eq:shape_meas} may be constructed for any trajectory $\ts_s = (y, s,  \omega_0, t_p)$ \rw{with $\omega_p = \omega(t_p \mid \omega_0), \ x_p = A(s; \omega_p)$}. One choice of these measures are $\mu_0 = \delta_{\omega =\omega_0}, \ \mu_p = \delta_{t=t_p} \otimes \delta_{\omega =\omega_p}, \ \eta = \delta_{x=x_p} \otimes \delta_{y =y}, \ \mu_s = \delta_{s=s} \otimes \delta_{\omega=\omega_p}$ and  $\mu$ \rev{as the \rw{occupation measure $t \mapsto (t, \omega(t \mid \omega_0^*)$ in times $[0, t^*_p]$.}} 
\rw{The feasible set of the constraints contains all trajectory-constructed measures, so  $p^* \leq P^*$.}
\end{proof}

\begin{lem}
\label{lem:shape_mass}
All measures in \eqref{eq:shape_meas} have bounded mass under Assumption A1'.
\end{lem}
\begin{proof}
This follows from the steps of Lemma \ref{lem:bounded}. The conditions hold that $1 = \inp{1}{\mu_0} = \inp{1}{\mu_p}$ \eqref{eq:shape_meas_liou}, $\inp{1}{\mu_p} = \inp{1}{\mu_s}$ \eqref{eq:shape_meas_marg_shape}, $\inp{1}{\mu_s} = \inp{1}{\eta}$ \eqref{eq:shape_meas_marg_coord}, and $\inp{1}{\mu} \leq T$ by \eqref{eq:shape_meas_liou}.
\end{proof}

}

\subsection{Shape Distance Function Program}

Defining a new dual function $z(\omega)$ against constraint \eqref{eq:shape_meas_marg_shape} (also observed in \eqref{eq:shape_demo_meas_marg_shape}), the  Lagrangian of problem \eqref{eq:shape_meas} is:
\begin{align}
    \label{eq:lagrangian_shape}
    \scL &= \inp{c(x, y)}{\eta} + \inp{v(t, x)}{ \delta_0 \otimes\mu_0 + \Lie_f^\dagger \mu - \mu_p} \nonumber\\
    &+ \inp{z(\omega)}{\pi^\omega_\#(\mu_p - \mu_s)} + \gamma(1 - \inp{1}{u_0}) \\
    &+ \inp{w(x)}{A(s;\omega)_\# \mu_s - \pi^x_\# \eta }. \nonumber
\end{align}

The Lagrangian in \eqref{eq:lagrangian_shape} may be manipulated into,
\begin{align}
    \label{eq:lagrangian_shape_2}
    \scL &=  \gamma + \inp{c(x, y)-w(x)}{\eta} + \inp{v(0, \omega)-\gamma}{\mu_0}\nonumber \\
    &+ \inp{\Lie_f v(t, \omega)}{\mu} + \inp{z(\omega) - v(t, \omega)}{\mu_p}  \\
    &+  \inp{w(A(s;\omega)) -z(\omega)}{\mu_s}.\nonumber
\end{align}
The dual of program \eqref{eq:shape_meas} provided by minimizing the Lagrangian \eqref{eq:lagrangian_shape_2} with respect to $(\eta, \mu_s, \mu_p, \mu, \mu_0)$ is,
\begin{subequations}
\label{eq:shape_cont}
\begin{align}
    d^* = & \ \rev{\sup}_{\gamma \in \R} \quad \gamma & \\
    & v(0, \omega) \geq \gamma&  &\forall x \in \Omega_0 \label{eq:shape_cont_init}\\
    & c(x, y)\geq w(x) & & \forall (x, y) \in X \times X_u \label{eq:shape_cont_cw}\\
    & w(A(s;\omega)) \geq z(\omega)  & &\forall (s, \omega) \in S \times \Omega \label{eq:shape_cont_wz}\\ 
    & z(\omega) \geq v(t, \omega) & &\forall (t, \omega) \in [0, T] \times \Omega \label{eq:shape_cont_zv}  \\
    & \Lie_f v(t, \omega) \geq 0 & &\forall (t, \omega) \in [0, T] \times \Omega \label{eq:shape_cont_f}\\
    &  w \in C(X), \ z \in C(\Omega) \\
    &v \in C^1([0, T] \times X). & \label{eq:shape_cont_v}
\end{align}
\end{subequations}

\rev{
\begin{theorem}
Problems \eqref{eq:shape_meas} and \eqref{eq:shape_cont} are strongly dual under assumptions A1'-A6'.
\end{theorem}
\begin{proof}
 This holds by extending the proof of Theorem \ref{thm:strong_duality_dist} found in Appendix \ref{app:duality} and applying Theorem 2.6 of \cite{tacchi2021thesis}.
\end{proof}
 
 \begin{remark}
 Program \ref{eq:shape_cont} imposes that a chain of lower bounds $v(t,\omega) \leq z(\omega) \leq w(A(s; \omega)) \leq c(A(s;\omega)), y)$ holds for all $(s, \omega, t, y) \in S \times \Omega \times [0, T] \times  X_u$ (similar in principle to Remark \ref{rmk:chain_lower}).
 \end{remark}

 \begin{theorem}
 Under assumptions A1'-A6', there exists a feasible $(\gamma, v, w, z)$ such that $P^* \rw{-} \delta \leq d^* \leq P^*$ between \eqref{eq:shape_cont} and \eqref{eq:shape_dist}.
 \end{theorem}
 \begin{proof}
 This proof follows from the steps of Theorem \ref{thm:relaxation_gap}. The following $C^0$ functions are feasible for constraints \eqref{eq:shape_cont_cw} and \eqref{eq:shape_cont_wz},
 \begin{align*}
     w(x) &= \inf_{y\in X_u} c(x,y), & 
     z(\omega) &= \inf_{(s,y) \in S \times X_u} c(A(s; \omega), y).
 \end{align*}
 The function $z(\omega)$ is $C^0$ since it is generated by the infimum of the composition of two $C^0$ functions $c$ and $A$.
 The auxiliary function $v(t, \omega)$ may be chosen to solve the peak minimization problem $\min_{t, \omega_0} z(\omega(t \mid \omega_0))$ along trajectories starting from $\Omega_0$ up to $\delta$-optimality (by the Flow map construction of \cite{fantuzzi2020bounding} as used in Theorem \ref{thm:relaxation_gap}).
 
 \end{proof}
 
\begin{cor}
Problem \eqref{eq:shape_dist} may be approximated up to $\delta$-optimality by smooth \rw{(polynomial)} functions $(v, w, z)$ under Assumptions A1'-A6'.
\end{cor} 
\begin{proof}
For any $\epsilon > 0$, Stone Weierstrass approximations $\bar{w} \in \R[x], \bar{z} \in \R[\omega]$ may be constructed such that 
\rw{$\norm{c(x; X_u) - \epsilon/2 - \bar{w}(x)}_{C^0(X)} \leq \epsilon/2$}
and 
\rw{$\norm{\bar{w}(A(s; \omega)) - \epsilon/2 - \bar{z}(\omega)}_{C^0(\Omega)} \leq \epsilon/2$ (similar to Corollary \ref{cor:smooth}).

A polynomial $\bar{v}$ may be derived from a $\delta'$-approximation to the peak minimization problem with objective $\bar{z}(\omega)$ (in which $\delta' < (2/5) \delta$), in the same manner that $\bar{v}$ in Corrolary \ref{cor:smooth} utilized an objective of $\bar{w}$. By ensuring that  $\epsilon$ satisfies
\begin{equation*}
    \epsilon < \delta'/\left(\max\left[2, 2T, 2T \max_i \norm{f_i(t, \omega)}_{[0, T] \times \Omega}  \right]\right),
\end{equation*}
then $(\bar{v}, \bar{w}, \bar{z}, \gamma=P^* - \delta)$ will be feasible for  \eqref{eq:shape_cont_init}-\eqref{eq:shape_cont_v}.}
\end{proof}

\begin{remark}
 We briefly note that the \ac{LMI} formulation of \eqref{eq:shape_meas} will converge to $P^*$ under assumptions A1'-A6' if all sets $[0,T], X, X_u, \Omega_0, \Omega, S$ are Archimedean and if $f(t, \omega) \in \R[t,\omega], \  A(s; \omega) \in \R[s, \omega]$ (from Theorem \ref{thm:lmi_convergence}). Constraint \eqref{eq:shape_demo_meas_marg_coord} induces a linear expression in moments for $(\eta, \mu_s)$ for each $\alpha \in \N^n: \ 
     \inp{x^\alpha}{\eta} = \inp{A(s; \omega)^\alpha}{\mu_s}$.
\end{remark} 
 }
  \begin{remark}
 \label{rmk:shape_degree}
  If \rw{$A(s; \omega)$ is polynomial with degree $\kappa$}, then the $d$-degree relaxation of problem \eqref{eq:shape_meas} involves moments of $\mu_s$ up to order $2\kappa d$. For a system with $N_\omega$ orientation states and $N_s$ shape variables, the size of the moment matrix for $\mu_s$ is then $\binom{N_s + N_\omega + \kappa d}{\kappa d}$. LMI constraints associated with $\mu_s$ can become bottlenecks to computation, surpassing the contributions of $\mu$ and $\eta$ as $k$ increases. 
 \end{remark}
 
 \rev{
 \begin{remark}
 \label{rmk:shape_merge}
Continuing the discussion Remark \ref{rmk:combine_mup_eta}, the measures $\mu_s(s, \omega)$ and $\eta(x, y)$ may be combined together into a larger measure $\eta_s(s, \omega, y) \in \Mp{S \times \Omega \times X_u}$ with objective $\inf \inp{c(A(s; \omega), y)}{\eta_s}$ and constraint $\pi^\omega_\# \mu_p = \pi^\omega_\# \eta_s$. The moment matrix for $\eta_s$ would have the generally intractable size $\binom{N_s + N_\omega + n + \kappa d}{\kappa d}$.
\end{remark}


}

\section{Numerical Examples}
\label{sec:examples}

All code was written in Matlab 2021a, and is publicly available at the link \url{https://github.com/Jarmill/distance}.
The \rw{\acp{SDP}} were formulated by Gloptipoly3 \cite{henrion2003gloptipoly} through a Yalmip interface \cite{lofberg2004yalmip}, and were solved using Mosek \cite{mosek92}. The experimental platform was an Intel i9 CPU with a clock frequency of 2.30 GHz and 64.0 GB of RAM.
The squared-$L_2$ cost $c(x, y) = \sum_i (x_i - y_i)^2$ is used in solving Problem \eqref{eq:dist_lmi} unless otherwise specified. The documented bounds are the square roots of the returned quantities, yielding lower bounds to the $L_2$ distance.

\subsection{Flow System with Moon}

The half-circle unsafe set in Figure \ref{fig:safety_distance} is a convex set. The moon-shaped unsafe set $X_u$ in Figure \ref{fig:flow_moon_collision} is the  nonconvex region outside the circle with radius $1.16$ centered at $(0.6596, 0.3989)$ and inside the circle with radius 0.8 centered at $(0.4, -0.4)$. The dotted red line demonstrates that trajectories of the Flow system would be deemed unsafe if $X_u$ was relaxed to its convex hull.

        \begin{figure}[ht]
        \centering
        \includegraphics[width=0.5\linewidth]{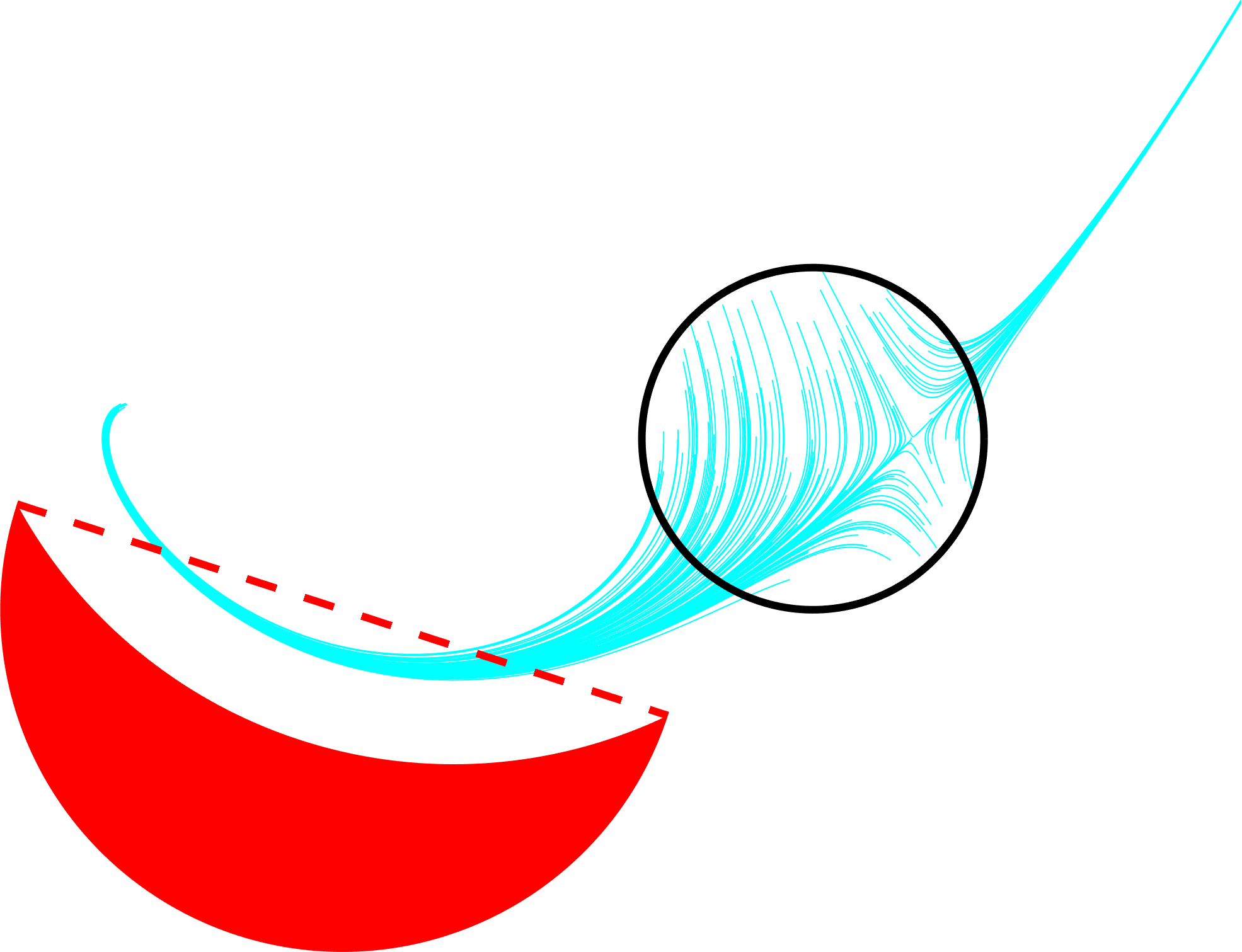}
        \caption{\label{fig:flow_moon_collision}Collision if $X_u$ is relaxed to its convex hull.}
    \end{figure}

The $L_2$ distance bound of $0.1592$ in Figure \ref{fig:flow_moon_dist} was found at the degree-5 relaxation of Problem \eqref{eq:dist_lmi} \rw{with $X=[-3,3]^2$}. The moment matrices $\M_d(m^0), \; \M_d(m^p), \M_d(m^\eta)$ at $d=5$ were approximately rank-1 and near-optimal trajectories were successfully extracted. This near-optimal trajectory starts at $x_0^* \approx (1.489, -0.3998)$ and reaches a closest distance between $x_p^* \approx (1.113, -0.4956)$ and $y^* \approx (1.161, -0.6472)$ at time $t_p^* \approx 0.1727$.
The distance bounds computed at the first five relaxations are $L^{1:5}_2 = [1.487 \times 10^{-4}, 2.433 \times 10^{-4}, 0.1501, 0.1592, 0.1592]$.

        \begin{figure}[ht]
        \centering
        \includegraphics[width=0.5\linewidth]{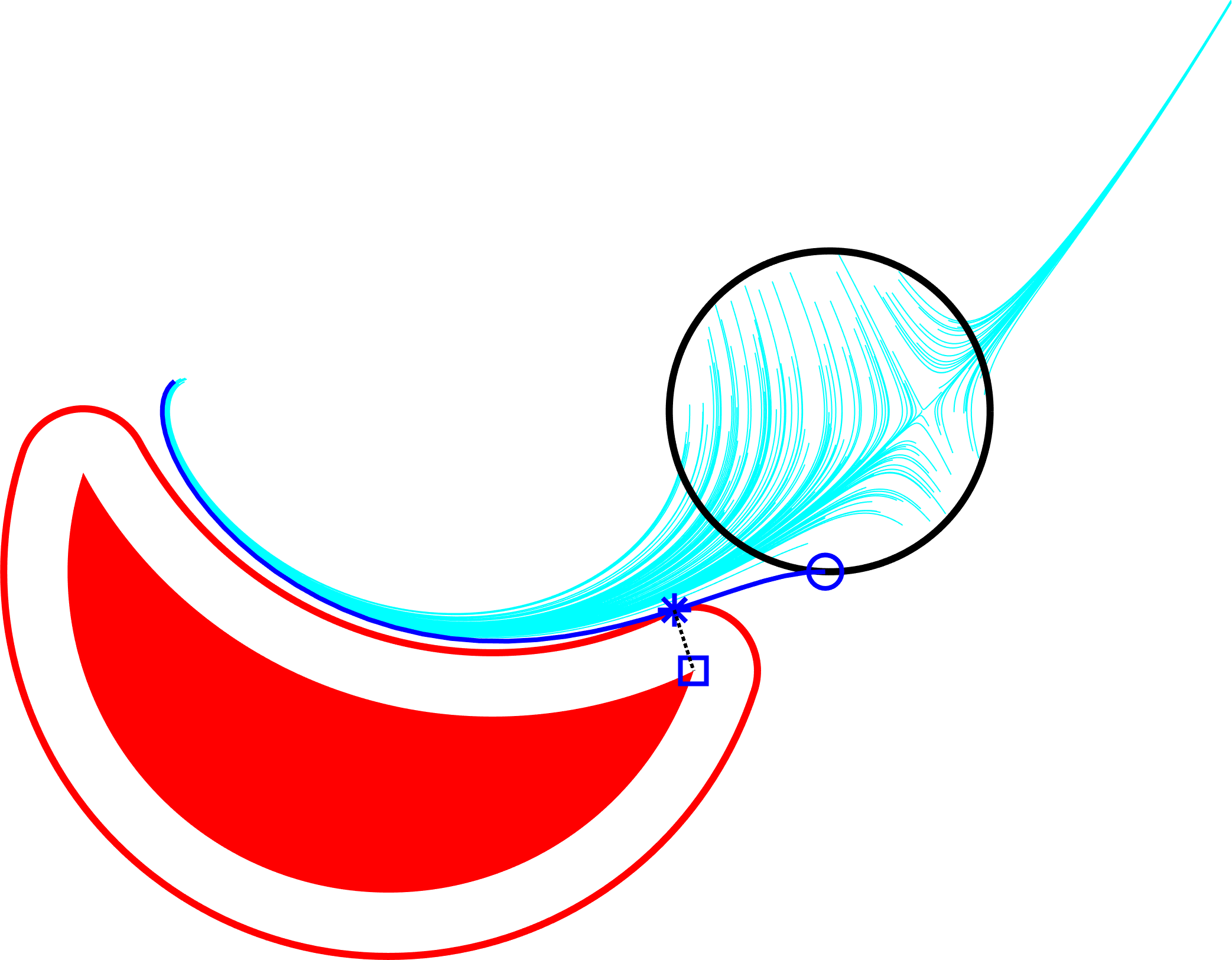}
        \caption{\label{fig:flow_moon_dist} $L_2$ bound of $0.1592$}
    \end{figure}

\subsection{Twist System}
\label{sec:twist}

    The Twist system is a three-dimensional dynamical system parameterized by matrices $A$ and $B$,
    \begin{align}
    \label{eq:twist_dynamics}
    \dot{x}_i(t) &= \textstyle \sum_j A_{ij} x_j - B_{ij}(4x_j^3 - 3x_j)/2,
    \end{align}
    \begin{align}
    \label{eq:twist_parameters}
    A &= \begin{bmatrix}-1 & 1 & 1\\ -1 &0 &-1\\ 0 & 1 &-2\end{bmatrix} &  B &=  \begin{bmatrix}-1 & 0 & -1\\ 0 &1 &1\\ 1 & 1 &0\end{bmatrix}.
       \end{align}
      
    The cyan curves in each panel of Figure \ref{fig:twist} are plots of trajectories of the Twist system between times $t \in [0, 5]$. These trajectories start at the $X_0 = \{x \mid(x_1+0.5)^2 + x_2^2 + x_3^2 \leq 0.2^2\}$ which is pictured by the grey spheres. The unsafe set $X_u = \{x \mid (x_1-0.25)^2 + x_2^2 + x_3^2 \leq 0.2^2, \ x_3 \leq 0\}$ is drawn in the red half-spheres. \rw{The underlying space is $X=[-1,1]^3$.}
    
    The red shell in Figure \ref{fig:twist_l2} is
   the cloud of points within an $L_2$ distance of $0.0427$ of $X_u$, as found through a degree 5 relaxation of \eqref{eq:dist_lmi}.
Figure \ref{fig:twist_l4} involves an $L_4$ contour of $0.0411$, also found at \rw{order 5}. The first few distance bounds for the $L_2$ distance are $L^{1:5}_2 = [0, 0, 0.0336, 0.0425, 0.0427]$, and for the $L_4$ distance are $L^{2:5}_4 = [0, 0.0298, 0.0408, 0.0413]$. 
Fourth degree moments are required for the $L_4$ metric, so the $L^{2:5}_4$ sequence starts at \rw{order} 2.

\begin{figure}[ht]
     \centering
     \begin{subfigure}[b]{0.48\linewidth}
         \centering
         \includegraphics[width=0.7\linewidth]{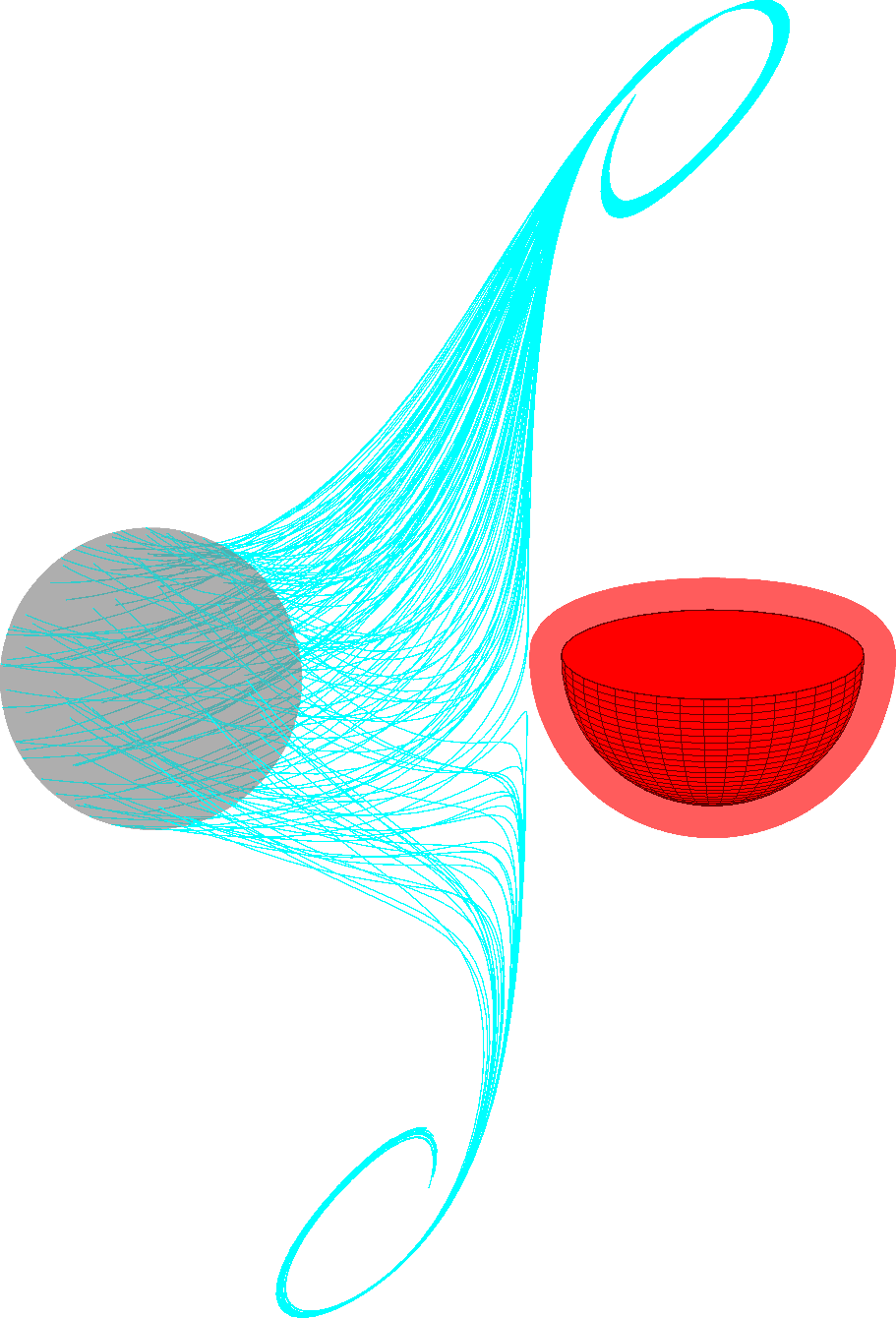}
         \caption{\label{fig:twist_l2} $L_2$ bound of $0.0427$}
         
     \end{subfigure}
     \;
     \begin{subfigure}[b]{0.48\linewidth}
         \centering
         \includegraphics[width=0.7\linewidth]{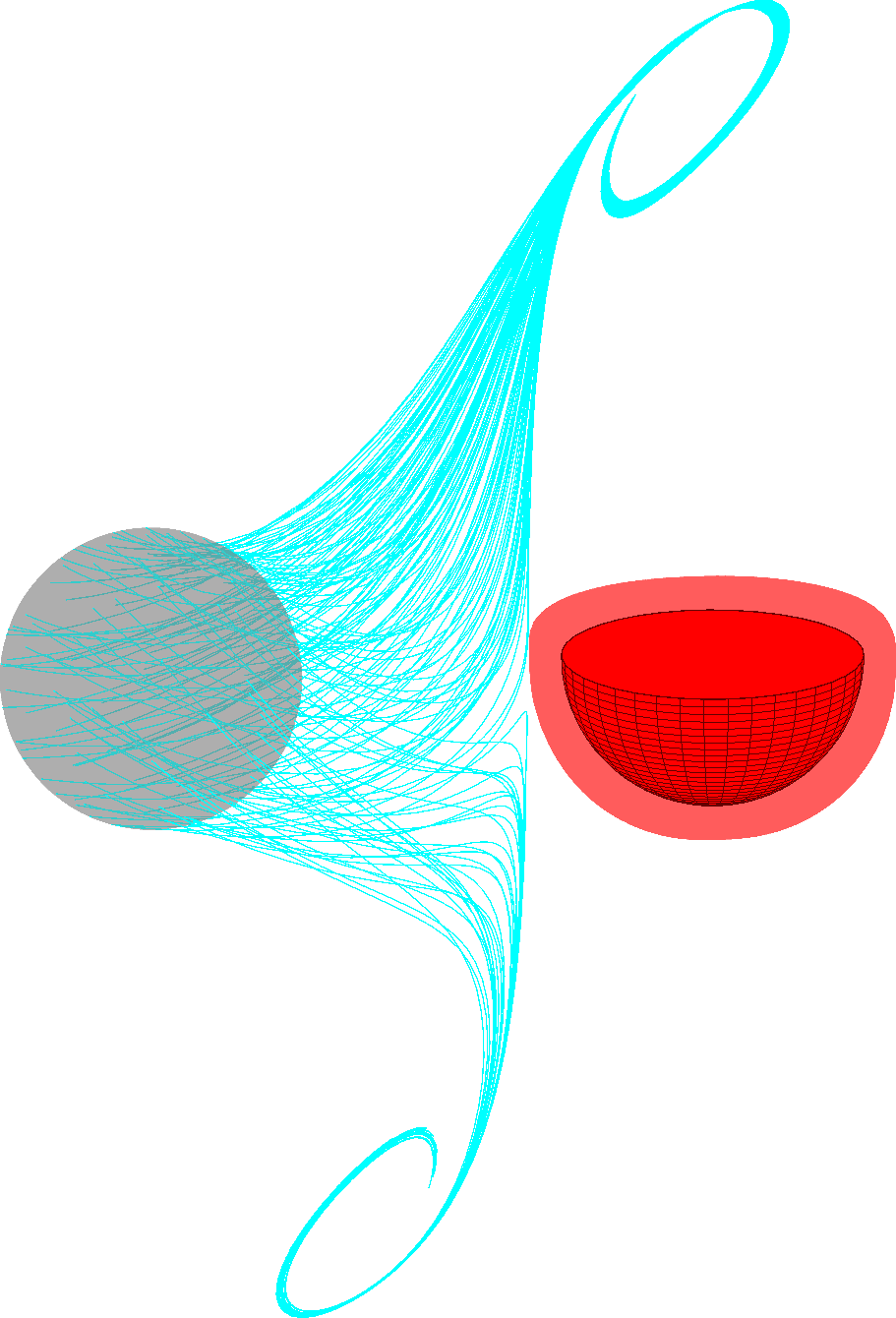}
         \caption{\label{fig:twist_l4}$L_4$ bound of $0.0411$}
     \end{subfigure}
      \caption{\label{fig:twist}Distance contours at order-5 relaxation for the  Twist system \eqref{eq:twist_dynamics}}
\end{figure}  


Table \ref{tab:twist_csp_l2} and \ref{tab:twist_csp_time} lists the $L_2$ bounds and runtimes respectively generated by a distance estimation task between trajectories and the half sphere of the above $L_2$ Twist system example. The high-degree relaxations (orders 4 and 5) are significantly faster as found by \rw{solving the \ac{SDP} associated with} the sparse \ac{LMI} (dual to the sparse \ac{SOS} with Putinar expression  \eqref{eq:csp_putinar}) as compared to the standard \rw{program} \eqref{eq:dist_lmi}. The certifiable $L_2$ bounds returned are roughly equivalent between relaxations.

\begin{table}[h]
    \centering
    \caption{\label{tab:twist_csp_l2} $L_2$ bounds for the Twist Example }
    \begin{tabular}{r|c c c c c}
         order &  2 & 3 & 4 & 5 & 6 \\
         Standard \ac{LMI} \eqref{eq:dist_lmi}& 0.000 & 0.0313 & 0.0425 & 0.0429 & 0.0429 \\ 
         Sparse \ac{LMI} with \eqref{eq:csp_putinar} & 
         0.000 & 0.0311 & 0.0424 & 0.0430 & 0.0429
    \end{tabular}
\end{table}

\begin{table}[h]
    \centering
    \caption{\label{tab:twist_csp_time} Time in seconds for the Twist Example}
    \begin{tabular}{r|c c c c c}
         order &  2 & 3 & 4 & 5 & 6 \\
         Standard \ac{LMI} \eqref{eq:dist_lmi}& 0.32 & 1.92 & 47.55 & 502.29 & 4028.94 \\ 
         Sparse \ac{LMI} with \eqref{eq:csp_putinar}&         0.31 & 1.19 & 7.07 & 45.89 & 184.42
    \end{tabular}
\end{table}

\subsection{Shape Examples}

Figure \ref{fig:shape_translate} visualizes a near-optimal trajectory of the shape distance estimation for orientations $\omega \in \R^2$ evolving as the flow system with an initial condition $\Omega_0 = \{\omega: (\omega_1-1.5)^2 + \omega_2^2 \leq 0.4^2\}$ \rw{in the space $\Omega: (\omega_1, \omega_2) \in [-3, 3]^2, \ \omega_3^2 + \omega_4^2 = 1$ (with a state set of $X = [-3, 3]^2$). }Suboptimal trajectories were suppressed in visualization to highlight the shape structure and attributes of the near-optimal trajectory. The degree-1 coordinate transformation function $A$ for pure translation with a constant rotation of $5\pi/12$ is,
\begin{equation}
    A(s; \omega) = \begin{bmatrix}\cos(5\pi/12) s_1 - \sin(5\pi/12) s_2 + \omega_1 \\ \cos(5\pi/12) s_1 + \sin(5\pi/12) s_2 + \omega_2\end{bmatrix}.
\end{equation}

This near-optimal trajectory with an $L_2$ distance bound of $0.1465$ was found at a degree-4 relaxation of Problem \eqref{eq:shape_meas}. The near-optimal trajectory is described by $\omega_0^* \approx (1.489, -0.3887)$,  $t_p^* \approx 3.090$, $\omega_p^* \approx (-0.1225, -0.3704)$,  $s^* \approx (-0.1, 0.1)$,   $x_p^* \approx (0, -0.2997)$, and $y^* \approx (-0.2261, -0.4739)$. The first five distance bounds are $L^{1:5}_2 = [1.205\times10^{-4},4.245\times10^{-4},0.1424,0.1465, 0.1465]$. 

    \begin{figure}[ht]
        \centering
        \includegraphics[width=0.5\linewidth]{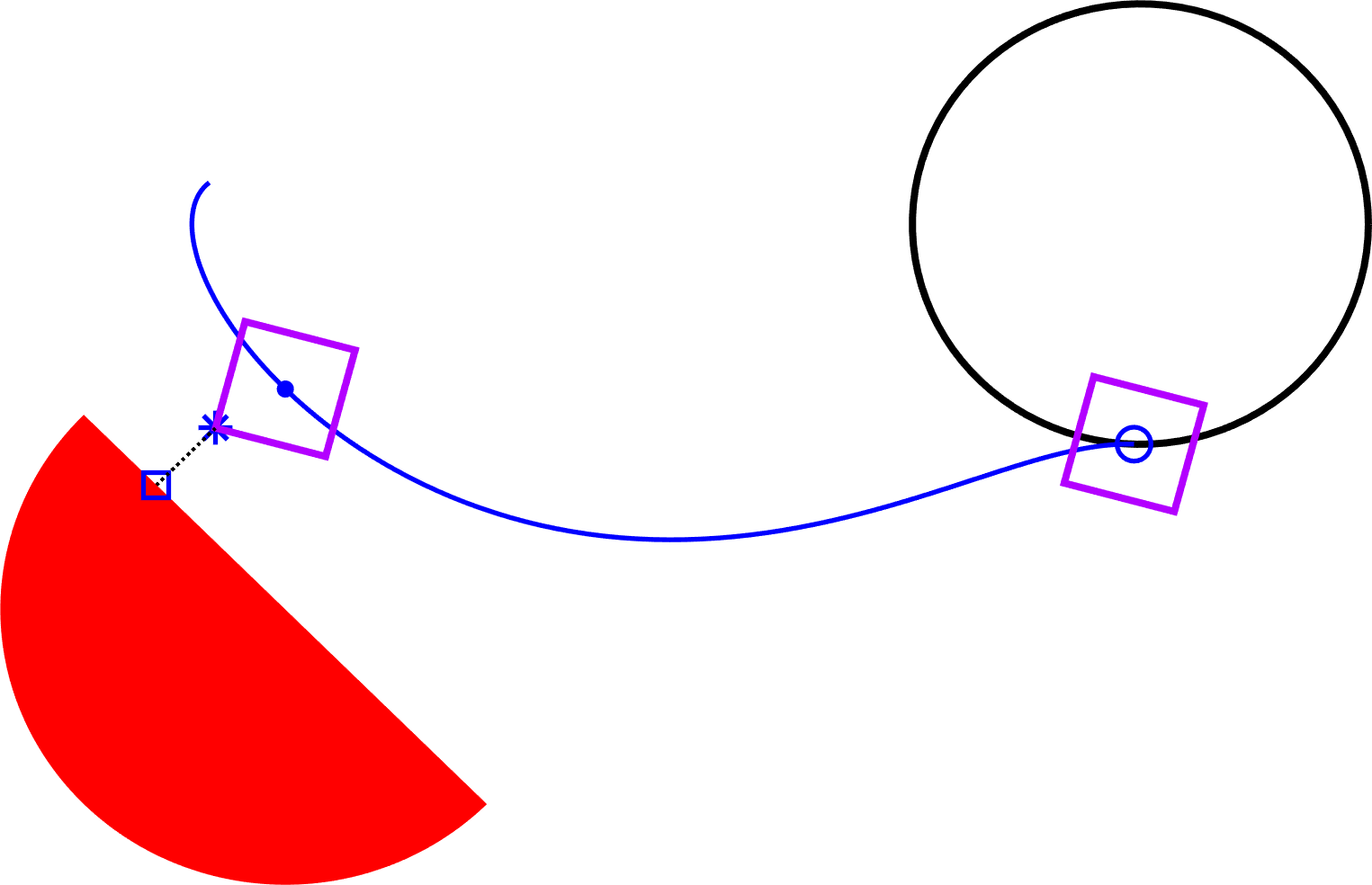}
        \caption{\label{fig:shape_translate} Translation, $L_2$ bound of $0.1465$}
    \end{figure}

\rev{In the following example, the shape $S$ is now rotating at an } angular velocity of 1 radian/second, as shown in the right panel of Fig. \ref{fig:shape_travel}.
The orientation $\omega \in SE(2)$ may be expressed as a semialgebraic lift through $\omega \in \R^4$ with trigonometric terms $\omega_3^2 + \omega_4^2 = 1$. The dynamics for this system are,
\begin{equation}
\label{eq:flow_rotate}
    \dot{\omega} = \begin{bmatrix}\omega_2 &  -\omega_1 -\omega_2 + \frac{1}{3}\omega^3_1 & -\omega_4 & \omega_3 \end{bmatrix}^T.
\end{equation}

The degree-2 coordinate transformation associated with this orientation is,
\begin{equation}
    A(s; \omega) = \begin{bmatrix}\omega_3 s_1 - \omega_4 s_2 + \omega_1 \\  \omega_3 s_1 + \omega_4 s_2 + \omega_2 \end{bmatrix}.
    \end{equation}

The shape measure $\mu_s \in \Mp{S \times \Omega}$ is distributed over 6 variables. The size of $\mu_s$'s moment matrix with $k = 2$ at degrees 1-4 is $[28, 210, 924, 3003]$. The first three distance bounds are $L^{1:3}_2 = [2.9158\times10^{-5}, 0.059162, 0.14255]$, and MATLAB runs out of memory on the experimental platform at degree 4. A successful recovery is achieved at the degree 3 relaxation, as pictured in Figure \ref{fig:shape_rotate}. This rotating-set near-optimal trajectory is encoded by $\omega_0^* \approx (1.575, -0.3928, 0.2588, 0.9659)$, $t_p^* \approx 3.371$, , $s^* \approx (-0.1, 0.1)$,  $x_p^* \approx (-0.1096, -0.3998)$, $\omega_p^* \approx (-0.0064, -0.2921, -0.0322, -0.9995)$, and $y^* \approx (-0.2104, -0.4896)$. Computing this degree-3 relaxation required 75.43 minutes.

    \begin{figure}[ht]
        \centering
        \includegraphics[width=0.5\linewidth]{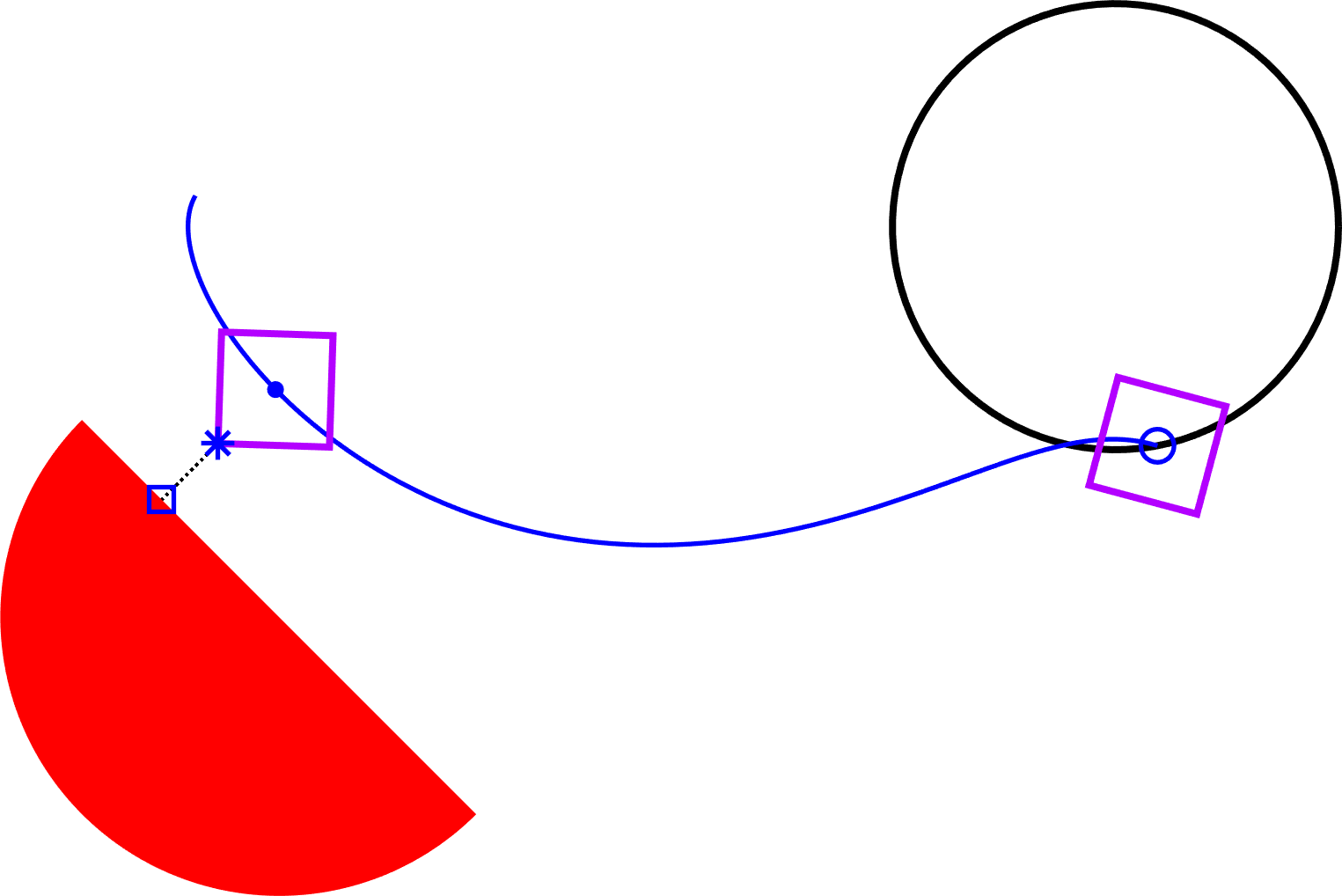}
        \caption{\label{fig:shape_rotate} Rotation, $L_2$ bound of $0.1425$}
    \end{figure}
\section{Extensions}
\label{sec:extensions}
This section presents modifications to the distance estimation programs in order to handle systems with uncertainties and distance functions $c$ generated by polyhedral norms.
\subsection{Uncertainty}


Distance estimation can be extended to systems with uncertainty. For the sake of simplicity, this section is restricted to time-dependent uncertainty. Assume that $H \subset \R^{N_h}$ is a compact set of plausible values of uncertainty, and the uncertain process $h(t), \forall t \in [0, T]$ may change arbitrarily in time within $H$ \cite{miller2021uncertain}.
The distance estimation problem with time-dependent uncertain dynamics is,
\begin{equation}
    \label{eq:dist_traj_w}
    \begin{aligned}
    P^* = & \rev{\inf}_{t,\:x_0, \:y, \: h(t)} c(x(t \mid x_0, h(t)), y) \\
    & \dot{x}(t) = f(t, x, h(t)), \ h(t) \in H & \quad \forall t \in [0, T] \\
    & x_0 \in X_0, \ y \in X_u.
    \end{aligned}
\end{equation}

The process $h(t)$ acts as an adversarial optimal control aiming to steer $x(t)$ as close to $X_u$ as possible. The occupation measure $\mu$ may be extended to a Young measure (relaxed control) $\mu \in \Mp{[0, T] \times X \times H}$ \cite{young1942generalized, henrion2008nonlinear}.

The Liouville equation \eqref{eq:dist_meas_liou} may be replaced by $\mu_p = \delta_0 \otimes \mu_0 + \pi^{tx}_\# \Lie^\dagger_f \mu$, which should be understood to read $\inp{v(t, x)}{\mu_p} = \inp{v(0, x)}{\mu_0} + \inp{\rev{\partial_t v(t,x) + f(t,x,h) \cdot \nabla_x v(t, x)}}{\mu}$ for all test functions $v \in C^1([0, T] \times X)$. 
\rev{Any trajectory with uncertainty process $h(t)$ may be represented by a tuple $(x_0, x_p, t_p, y, h(\cdot))$. This trajectory admits a measure representation similar to the proof of \ref{thm:meas_lower}, where the measure $\mu$ is \rw{the occupation measure of $ t \mapsto (t, x(t \mid x_0), h(t)))$ in times $[0,t_p]$.}}
The work in \cite{miller2021uncertain} applies a collection of existing uncertainty structures to peak estimation problems (time-independent, time-dependent, switching-type, box-type), and all of these structures may be applied to distance estimation.

To illustrate these ideas, consider the following Flow system with time-dependent uncertainty:
\begin{align}
\label{eq:flow_uncertain}
    \dot{x} &= \begin{bmatrix}
    x_2 \\
    (-1+h)x_1 -x_2 + \frac{1}{3}x_1^3
    \end{bmatrix} & h \in [-0.25, 0.25].
\end{align}

An $L_2$ distance bound of $0.1691$ is computed at the degree 5 relaxation of the uncertain distance estimation program, as visualized in Figure \ref{fig:flow_uncertain}. The first five distance bounds are $L_2^{1:5} = [5.125\times10^{-5}, 1.487\times10^{-4}, 0.1609, 0.1688, 0.1691]$.

        \begin{figure}[ht]
        \centering
        \includegraphics[width=0.45\linewidth]{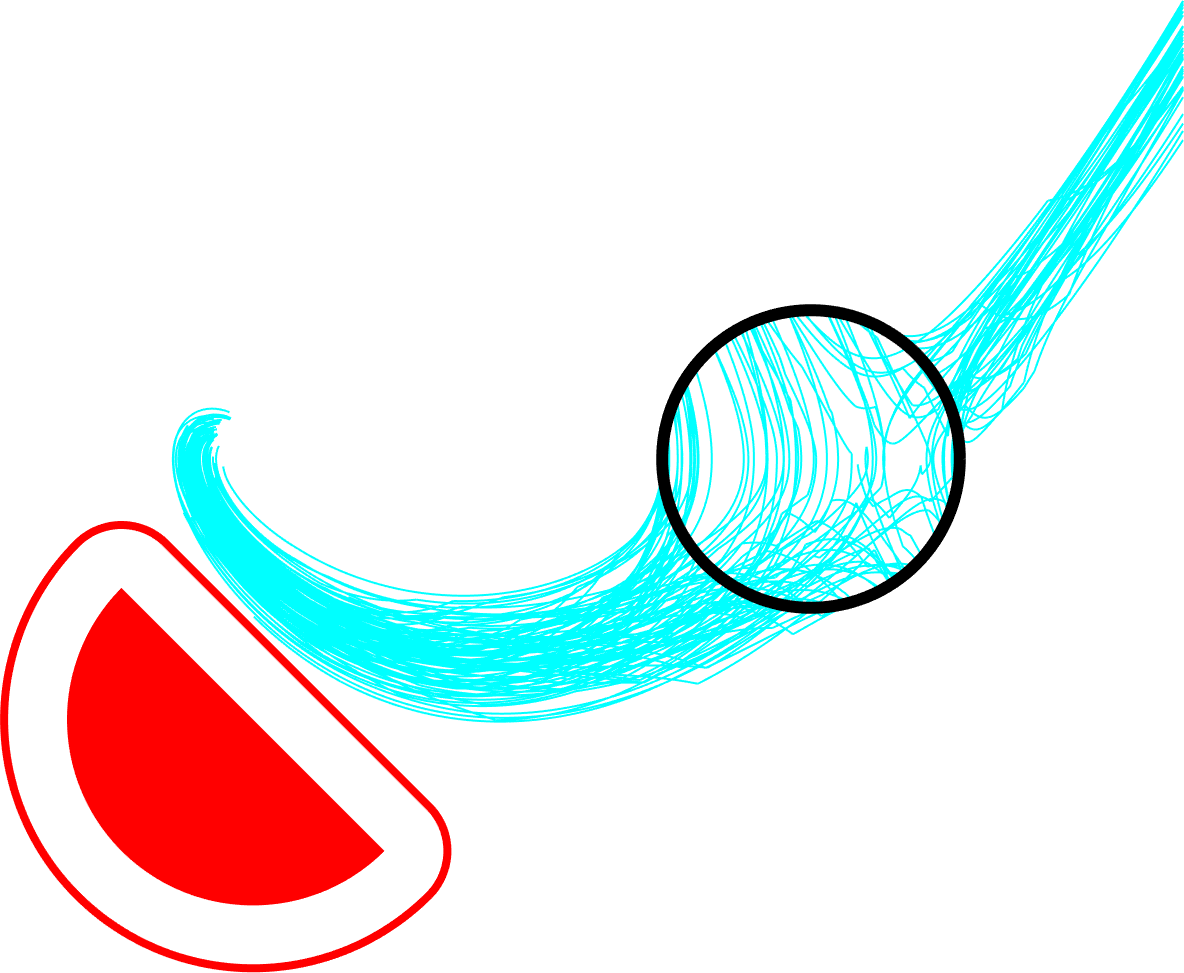}
        \caption{\label{fig:flow_uncertain} Uncertain Flow \eqref{eq:flow_uncertain}, $L_2$ bound of $0.1691$
         }
        \label{fig:my_label}
    \end{figure}

\subsection{Polyhedral Norm Penalties}
\label{sec:polyhedral}
The infinite dimensional \ac{LP} \eqref{eq:dist_meas} is valid for all continuous costs $c(x, y) \in C(X^2)$, but its LMI relaxation can only handle polynomial costs $c(x, y) \in \R[x, y]$. The $L_p$ distance is defined as $\norm{x-y}_p = \sqrt[\leftroot{-3}\uproot{3}p]{\sum_i \abs{x_i-y_i}^p}$ when $p$ is finite and $\norm{x-y}_\infty = \max_i \abs{x_i-y_i}$ for $p$ infinite. The \rev{cost $\norm{x-y}_p^p$} is polynomial when $p$ is finite and even\rev{;} otherwise the $L_p$ distance has a piecewise definition in terms of absolute values. The theory of convex (\ac{LP}) lifts may be used to interpret piecewise constraints into valid \acp{LMI} \cite{yannakakis1991expressing, gouveia2013lifts}. Slack variables $q \in \R$ (or $q_i \in \R$ as appropriate) may be added to form enriched infinite dimensional \acp{LP}. The objective $\inp{c}{\eta}$ from \eqref{eq:dist_meas_obj} could be replaced by the following terms for the examples of $L_\infty, L_1, $ and $L_3$ distances:
\begin{subequations}
\label{eq:distance_slack_obj}
\begin{align}
        \norm{x-y}_\infty & & &\qquad  \qquad \min \quad  q \\
        & & & -q \leq \inp{x_i-y_i}{\eta} \leq q & & \forall i=1, \ldots, n \nonumber\\[1.25\baselineskip]
        \norm{x-y}_1 & & &\qquad  \qquad \min \quad  \textstyle \sum_i q_i \\
        & & & -q_i \leq \inp{x_i-y_i}{\eta} \leq q_i & & \forall i=1, \ldots, n, \nonumber \label{eq:distance_slack_obj_l1}\\[1.25\baselineskip]
        \norm{x-y}_3^3 & & &\qquad  \qquad \min \quad  \textstyle \sum_i q_i \\
        & & & -q_i \leq \inp{(x_i-y_i)^3}{\eta} \leq q_i & & \forall i=1, \ldots, n. \nonumber
    \end{align}
\end{subequations}

Distances induced by polyhedral norms can be included through this lifting framework \cite{anderson1976discrete}. Figure \ref{fig:flow_l1} visualizes the near-optimal trajectory for a minimum $L_1$ distance bound of $0.4003$ (cost \eqref{eq:distance_slack_obj_l1}) at degree $4$. This trajectory starts at $x_0^* \approx (1.489, -0.3998)$ and reaches the closest approach between $x_p^* \approx (0, -0.2997)$ and $y^* \approx (-0.1777, -0.5223)$ at time $t^* \approx 0.6181$ units. The first five $L_1$ distance bounds are $L^{1:5}_1 = [3.179\times10^{-9},4.389\times10^{-8},0.3146,0.4003, 0.4003]$.

    \begin{figure}[ht]
    \centering
        \includegraphics[width=0.5\linewidth]{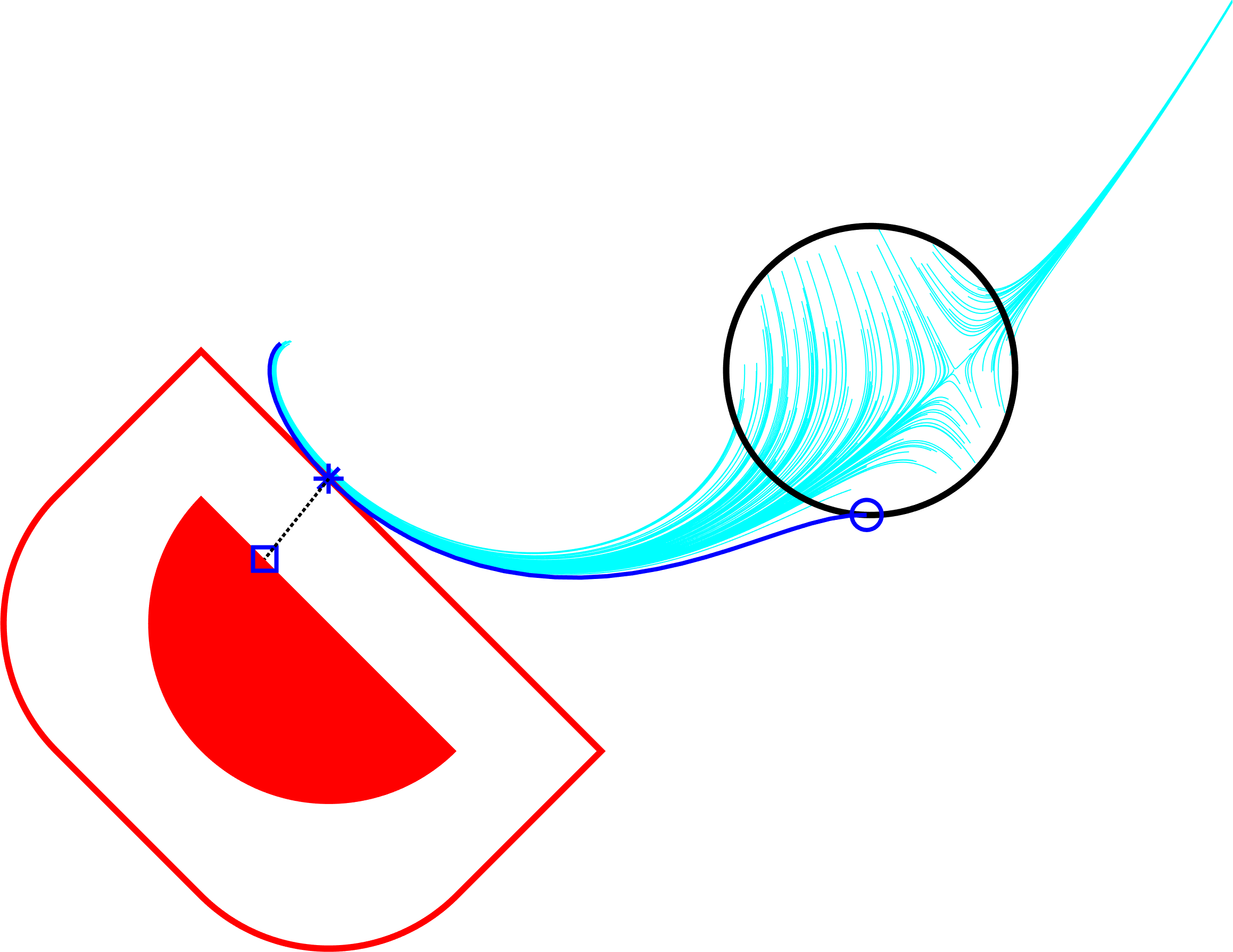}
        \caption{ \label{fig:flow_l1}$L_1$ bound of $0.4003$}
    \end{figure}

\section{Conclusion}
\label{sec:conclusion}

This paper presented an infinite dimensional linear program in occupation measures to approximate the distance estimation problem. The \ac{LP} objective is \rev{equal} to the distance of closest approach between points along trajectories and points on the unsafe set under mild \rev{compactness and regularity}  conditions.
Finite-dimensional truncations of this \ac{LP} yield a converging sequence of \ac{SDP} lower bounds to the minimal distance \rev{under further conditions (Archimedean)}. The distance estimation problem can be modified to accommodate dynamics with uncertainty, piecewise distance functions, and movement of shapes along trajectories. Future work includes formulating and implementing control policies to maximize the distance of closest approach to the unsafe set \rev{while still reaching a terminal set within a specified time}.

\appendix

\section{Proof of Strong Duality in Theorem \ref{thm:strong_duality_dist}}


\label{app:duality}
This proof will follow the method used in \rev{Theorem 2.6 of  \cite{tacchi2021thesis}} to prove duality. 


The two programs \rev{\eqref{eq:dist_meas} and \eqref{eq:dist_cont}} will be posed as a pair of standard-form infinite dimensional 
LPs \rev{using notation from \cite{tacchi2021thesis}}.
The following spaces may be defined:
\begin{align}
    \rev{\mathcal{X}}' &= C(X_0) \times C([0, T]\times X)^2 \times C(X \times X_u) \label{eq:dual_spaces}\\
    \rev{\mathcal{X}} &= \mathcal{M}(X_0) \times \mathcal{M}([0, T]\times X)^2 \times \mathcal{M}(X \times X_u). \nonumber
\end{align}
The nonnegative subcones of $\rev{\mathcal{X}}'$ and $\rev{\mathcal{X}}$ respectively are,
\begin{align}
    \rev{\mathcal{X}}_+' &= C_+(X_0) \times C_+([0, T]\times X)^2 \times C_+(X \times X_u) \label{eq:dual_cones}\\
    \rev{\mathcal{X}}_+ &= \Mp{X_0} \times \Mp{[0, T]\times X}^2 \times \Mp{X \times X_u}. \nonumber
\end{align}

The cones $\rev{\mathcal{X}}_+'$ and $\rev{\mathcal{X}}_+$ in \eqref{eq:dual_cones} are topological duals \rev{under assumption A1}, and the measures from \eqref{eq:dist_meas_joint}-\eqref{eq:dist_meas_occ} satisfy $\rev{\boldsymbol{\mu}} = (\mu_0, \mu_p, \mu, \eta) \in \rev{\mathcal{X}}_+$.
The spaces $\rev{\mathcal{Y}}$ and $\rev{\mathcal{Y}}'$ may be defined as,
\begin{align}
    \rev{\mathcal{Y}}' &= C(X) \times C^1([0, T] \times X) \times \R \\
    \rev{\mathcal{Y}} &= \mathcal{M}(X) \times C^1([0, T] \times X)' \times 0.
\end{align}
\rev{We express $\rev{\mathcal{Y}}_+ = \rev{\mathcal{Y}}$ and $\rev{\mathcal{Y}}_+' = \rev{\mathcal{Y}}'$ to maintain a convention with \cite{tacchi2021thesis} given there are no affine-inequality constraints in \eqref{eq:dist_meas}.} \rw{We equip $\mathcal{X}$ with the weak-* topology and $\mathcal{Y}$ with the (sup-norm bounded) weak topology.}
The arguments $\rev{\boldsymbol{\ell}} = (w, v, \gamma)$ from problem \eqref{eq:dist_cont} are members of the set $\rev{\mathcal{Y}}_+'$.

The linear operators $\A': \rev{\mathcal{Y}}_+' \rightarrow \rev{\mathcal{X}}_+'$ and $\A: \rev{\mathcal{X}}_+ \rightarrow \rev{\mathcal{Y}}_+$ induced from constraints \eqref{eq:dist_meas_marg}-\eqref{eq:dist_meas_prob} may be defined as,
\begin{align}
    \A(\rev{\boldsymbol{\mu}}) =[&\pi^{x}_\# \mu_p -\pi^{x}_\# \eta, \delta_0 \otimes\mu_0 + \Lie_f^\dagger \mu - \mu_p, \inp{1}{\mu_0}]\\ 
    \A'(\rev{\boldsymbol{\ell}}) = [&v(0,x)-\gamma, w(x)-v(t,x), \Lie_f v(t,x), 
    \rw{-}w(x)].\nonumber
\end{align}

The last pieces needed to convert \eqref{eq:dist_meas} into a standard-form LP are the cost vector $\rev{\mathbf{c}} = [0, 0, 0, c(x, y)]$ and the answer vector $\rev{\mathbf{b}} = [0, 0, 1] \in \rev{\mathcal{Y}}'$. 
Problem \eqref{eq:dist_meas} is therefore equivalent to (with $\inp{\rev{\mathbf{c}}}{\rev{\boldsymbol{\mu}}} = \inp{c}{\eta}$),
\begin{align}
    p^* =& \inf_{\rev{\boldsymbol{\mu}} \in \rev{\mathcal{X}}_+} \inp{\rev{\mathbf{c}}}{\rev{\boldsymbol{\mu}}} & & \rev{\mathbf{b}} - \A(\rev{\boldsymbol{\mu}}) \in \rev{\mathcal{Y}}_+. \label{eq:dist_meas_std}\\
\intertext{The dual LP to \eqref{eq:dist_meas_std} in standard form is (with $\inp{\rev{\boldsymbol{\ell}}}{\rev{\mathbf{b}}} = \gamma$),}
    d^* = &\sup_{\rev{\boldsymbol{\ell}} \in \rev{\mathcal{Y}}'_+} \inp{\rev{\boldsymbol{\ell}}}{\rev{\mathbf{b}}}
    & &\A'(\rev{\boldsymbol{\ell}}) - \rev{\mathbf{c}} \in \rev{\mathcal{X}}_+. \label{eq:dist_cont_std}
\end{align}

The operators $\A$ and $\A'$ are adjoints with $\inp{\A(\rev{\boldsymbol{\ell}})}{\rev{\boldsymbol{\mu}}} = \inp{\rev{\boldsymbol{\ell}}}{\A'(\rev{\boldsymbol{\mu}})}$ for all $\rev{\boldsymbol{\ell}} \in \rev{\mathcal{Y}}_+'$ and $\rev{\boldsymbol{\mu}} \in \rev{\mathcal{X}}_+$. 

\rev{
The sufficient conditions for strong duality and attainment of optimality between \eqref{eq:dist_meas_std} and  \eqref{eq:dist_cont_std} as outlined in Theorem 2.6 of \cite{tacchi2021thesis} are that:
\begin{enumerate}
    \item[R1] All support sets are compact (A1)
    \item[R2] All measure solutions have bounded mass (Lemma \ref{lem:bounded})
    \item[R3] All functions involved in the definitions of  $c$ and $\A$ are continuous (A2, A3)
    \item[R4] There exists a $\rev{\boldsymbol{\mu}}_{\textrm{feas}} \in \rev{\mathcal{X}}_+$ with $\rev{\mathbf{b}} - \A(\rev{\boldsymbol{\mu}}_\textrm{feas}) \in \rev{\mathcal{Y}}_+$ 
\end{enumerate}

The requirements R1 and R2 hold by Assumption A1 and Lemma \ref{lem:bounded} respectively. R3 is valid given that $c(x, y)$ is $C^0$ (A3), the projection map $\pi^{x}$ is continuous, and the mapping $(t,x) \mapsto \Lie_f v(t,x)$ is $C^0$ for $v\in C^1$ and $f$ Lipschitz (continuous) (A2). A feasible measure $\rev{\boldsymbol{\mu}}_\textrm{feas}$ may be constructed from the process in Theorem \ref{thm:meas_lower} from a tuple $\mathcal{T}$, therefore satisfying R4.

Strong duality between \eqref{eq:dist_meas} and \eqref{eq:dist_cont} is therefore proven after satisfaction of all four requirements.

}
\section{Moment-SOS Hierarchy}
\label{sec:moment_sos}

 
The standard form for a measure \ac{LP} with variable $\mu \in \Mp{X}$ involving a cost function $p \in C(X)$ and a (possibly infinite) set of affine constraints $\inp{a_j}{\mu} = b_j$ with $b_j \in \R$ and $a_j \in C(X)$ for $j = 1, \ldots, J_{max}$ is, 
\begin{subequations}
\label{eq:meas_program}
\begin{align}
    p^* =&  \sup_{\mu \in \Mp{X}} \inp{p}{\mu}\\
    &\inp{a_j(x)}{\mu} = b_j  & & \forall j=1,\ldots,J_{max}. \label{eq:meas_program_constraints}
\end{align}
\end{subequations}

The dual problem to Program \eqref{eq:meas_program} with dual variables $v_j \in \R: \forall j = 1, \ldots, m$ is,
\begin{subequations}
\label{eq:cont_program}
\begin{align}
    d^* =&  \inf_{v \in \R^m} \textstyle\sum_j b_j v_j\\
    &p(x) - \textstyle\sum_j a_j(x) v_j \geq 0  & & \forall x \in X. \label{eq:cont_nonneg}
\end{align}
\end{subequations}

The objectives in  \eqref{eq:meas_program} and \eqref{eq:cont_program} will match ($p^*=d^*$ strong duality) if $p^*$ is \rev{finite} and if the mapping $\mu \rightarrow \{\inp{a_j(x)}{\mu}\}_{j=1}^m$ is closed in the weak-* topology (Theorem 3.10 in \cite{anderson1987linear}).

When $p(x)$ and all $a_j(x)$ are polynomial, constraint \eqref{eq:cont_nonneg} is a polynomial nonnegativity constraint. The restriction that a polynomial $q(x) \in \R[x]$ is nonnegative over $\R^n$ may be \rev{strengthened} to finding a set of polynomials $\{q_i(x)\}$ such that $q(x) = \sum_i q_i(x)^2$. The polynomials $\{q_i(x)\}$ are \iac{SOS} certificate of nonnegativity of $q(x)$, given that the square of a real quantity $q_i(x)$ at each $i$ and $x$ is nonnegative. The set of \ac{SOS} polynomials in indeterminate quantities $x$ is expressed as $\Sigma[x]$,
\rev{with a maximal-degree-$d$ subset of $\Sigma[x]_{\leq d}$}.

The quadratic module $Q[g]$ formed by the constraints describing the basic semialgebraic set $\K = \{x \mid g_i(x) \geq 0, \ i = 1, \ldots, N_c\}$ is the set of polynomials:
\begin{align}Q[g] = \left\{\sigma_0(x) + \textstyle \sum_{i=1}^{N_c} {\sigma_i(x)g_i(x)}  \right\}, \label{eq:quad_module}
        \end{align}
such that the multipliers $\sigma$ are SOS,
\begin{equation}
    \sigma_i(x) \in \Sigma[x] \qquad \forall i = 0, \ldots, N_c.
\end{equation}
The basic semialgebraic set $\K$ is compact if there exists a constant $0 \leq R < \infty$ such that $\K$ is contained in the ball $R \leq \norm{x}_2^2$. $\K$ satisfies the Archimedean property if the polynomial $R - \norm{x}_2^2 $ is a member of $Q[g]$. The Archimedean property is stronger than compactness \cite{cimpric2011quadratic}, and compact sets may be rendered Archimedean by adding a redundant ball constraint $R - \norm{x}_2^2 \geq 0$ to the list of constraints describing in $\K$ \rev{(though finding such an $R$ may be difficult)}.
\rw{When $\K$ is Archimedean, every polynomial satisfying $p(x) > 0, \forall x \in \K$ has a representation (Putinar's Positivestellensatz  \cite{putinar1993compact}):} 
\begin{equation}
\label{eq:putinar}
    \begin{aligned}
        & p(x) = \sigma_0(x) + \textstyle \sum_i {\sigma_i(x)g_i(x)}\\
        & \sigma_0(x) \in \Sigma[x] \qquad \sigma_i(x) \in \Sigma[x].
    \end{aligned}
\end{equation}

\rev{The \ac{WSOS} set $\Sigma[\K]$ is the set of polynomials that admit a positivity certificate over $\K$ from \eqref{eq:putinar}. Its maximal degree-$d$ subset is $\Sigma[\K]_{\leq d}$).}
Given a multi-index $\alpha \in \N^n$, the $\alpha$-moment of a measure $\mu \in \Mp{X}$ is $\rev{\bm}_\alpha = \inp{x^\alpha}{\mu}$. 
\rev{
An infinite moment matrix $\M[\rbm]_{\alpha,\beta} = \bm_{\alpha + \beta}$ indexed by monomials $\alpha, \beta \in \N^n$ may be constructed from the moment sequence $\bm$. 
}

The degree-$d$ moment matrix $\M_d[\rbm]$ of size $\binom{n+d}{d}$ is the submatrix of  $\M[\rbm]$ where the indices $\M_d[\rbm]_{\alpha, \beta}$ have total degree bounded by $0 \leq \abs{\alpha}, \abs{\beta} \leq d$. Given a polynomial $g(x) \in \R[x]$, the localizing matrix associated with $g$ is a square infinite-dimensional symmetric matrix with entries $\M[g \rbm]_{\alpha, \beta} = \textstyle\sum_{\gamma \in \N^n} g_{\gamma} \rbm_{\alpha+ \beta + \gamma}$. A moment sequence $m$ has a representing measure $\mu \in \Mp{\K}$ if there exists  $\mu$ supported in $\K$ such that $\rev{\bm}_\alpha = \inp{x^\alpha}{\mu} \ \forall \alpha \in \N^n$. 
The \ac{LMI} conditions that $\M[\rbm] \succeq 0$ and $\M[g_i \rbm] \succeq  0 \ \forall i = 1, \ldots, N_c$ are necessary to guarantee the existence of a representing measure associated with $\rev{\bm}$. \rev{The moment matrix $\M[\bm]$ is a localizing matrix with the function $g=1$.} These \ac{LMI} conditions are sufficient if the set $\K$ is Archimedean, and all compact sets may be rendered Archimedean through the application of a redundant ball constraint  \cite{putinar1993compact}. 

Assume that each polynomial $g_i(x)$ in the constraints of $\K$ has a degree $d_i$. \rev{We define a block-diagonal matrix $\M_d[\K \bm]$ containing the moment and all localizing matrices as, 
\begin{equation}
\label{eq:moment_weight}
     \diag{\M_d[\rev{\bm}], \{ \M_{d - d_i}(g_{i} \rev{\bm}) \  \forall i = 1, \ldots, N_c\}}.
\end{equation}
}
The degree-$d$ moment relaxation of Problem \eqref{eq:meas_program} with variables $y \in \R^{\binom{n+2d}{2d}}$ 
is,
\begin{subequations}
\begin{align}
\label{eq:mom_program}
    p^*_d &=  \ \rev{\max}_{\rbm} \textstyle\sum_\alpha p_\alpha \rev{\bm}_\alpha, & &  \M_d[\rev{\K \bm}] \succeq 0  \\
    & \textstyle\sum_\alpha a_{j \alpha} \rev{\bm}_\alpha\ = b_j  & & \forall j = 1, \ldots, m.
\end{align}
\end{subequations}

The bound $p^*_d \geq p^*$ is an upper bound for the infinite-dimensional measure \ac{LP}. The decreasing sequence of upper bounds $p_d^* \geq p_{d+1}^* \geq \ldots \geq p^*$ is convergent to $p^*$ as $d \rightarrow \infty$ if $\K$ is Archimedean. 
The dual semidefinite program to \eqref{eq:mom_program} is the degree-$d$ \ac{SOS} relaxation of \eqref{eq:cont_program}:
\begin{subequations}
\label{eq:sos_program}
\begin{align}
    d^*_d =&  \rev{\min}_{v \in \R^m} \textstyle\sum_j b_j v_j\\
    &p(x) -\textstyle \sum_j a_j(x) v_j = \sigma_0(x) + \textstyle\sum_k {\sigma_i(x)g_i(x)}\\
        & \sigma(x) \in \Sigma[x]_{\rev{\leq 2d}} \\
        & \sigma_i(x) \in \Sigma[x]_{\rev{\leq 2d - \lceil \deg g_i/2 \rceil}} \quad \forall i \in 1,\ldots, N_c. 
\end{align}
\end{subequations}
\rev{We use the convention that the degree-$d$ SOS tightening of \eqref{eq:sos_program} involves polynomials of maximal degree $2d$.}
When the moment sequence $\rev{\bm}_\alpha$ is bounded $(\abs{\rev{\bm}_\alpha} < \infty \ \forall \abs{\alpha} \leq 2d)$ and there exists an interior point of the affine measure constraints in \eqref{eq:meas_program_constraints}, then the finite-dimensional truncations 
\eqref{eq:mom_program} and \eqref{eq:sos_program} will also satisfy strong duality $p^*_k = d^*_k$ 
\rev{at each degree $k$} 
(by arguments from Appendix D/Theorem 4 of \cite{henrion2013convex}
using Theorem 5 of \cite{trnovska2005strong}, \rev{also applied in Corrolary 8 of \cite{tacchi2022convergence} }).
The sequence of upper bounds (outer approximations) $p_d^* \geq p_{d+1}^* \geq \ldots $ computed from \rw{\acp{SDP}} is called the Moment-\ac{SOS} hierarchy.
  
\section*{Acknowledgements}

The authors would like to thank Didier Henrion, Victor Magron, \rw{Matteo Tacchi,} and the POP group at LAAS-CNRS for many technical discussions and suggestions. \rw{We are also grateful to the anonymous reviewers for their many suggestions to improve the original manuscript.}



\bibliographystyle{IEEEtran}
\bibliography{wass_reference.bib}

\begin{thebibliography}{10}
\providecommand{\url}[1]{#1}
\csname url@samestyle\endcsname
\providecommand{\newblock}{\relax}
\providecommand{\bibinfo}[2]{#2}
\providecommand{\BIBentrySTDinterwordspacing}{\spaceskip=0pt\relax}
\providecommand{\BIBentryALTinterwordstretchfactor}{4}
\providecommand{\BIBentryALTinterwordspacing}{\spaceskip=\fontdimen2\font plus
\BIBentryALTinterwordstretchfactor\fontdimen3\font minus
  \fontdimen4\font\relax}
\providecommand{\BIBforeignlanguage}[2]{{%
\expandafter\ifx\csname l@#1\endcsname\relax
\typeout{** WARNING: IEEEtran.bst: No hyphenation pattern has been}%
\typeout{** loaded for the language `#1'. Using the pattern for}%
\typeout{** the default language instead.}%
\else
\language=\csname l@#1\endcsname
\fi
#2}}
\providecommand{\BIBdecl}{\relax}
\BIBdecl

\bibitem{boyd1994linear}
S.~Boyd, L.~El~Ghaoui, E.~Feron, and V.~Balakrishnan, \emph{Linear {M}atrix
  {I}nequalities in {S}ystem and {C}ontrol {T}heory}.\hskip 1em plus 0.5em
  minus 0.4em\relax SIAM, 1994, vol.~15.

\bibitem{lasserre2009moments}
J.~B. Lasserre, \emph{{Moments, Positive Polynomials And Their}
  {Applications}}, ser. Imperial College Press Optimization Series.\hskip 1em
  plus 0.5em minus 0.4em\relax World Scientific Publishing Company, 2009.

\bibitem{helmes2001computing}
K.~Helmes, S.~R{\"o}hl, and R.~H. Stockbridge, ``{Computing Moments of the Exit
  Time Distribution for Markov Processes by Linear Programming},''
  \emph{Operations Research}, vol.~49, no.~4, pp. 516--530, 2001.

\bibitem{cho2002linear}
M.~J. Cho and R.~H. Stockbridge, ``{Linear Programming Formulation for Optimal
  Stopping Problems},'' \emph{SIAM J. Control Optim.}, vol.~40, no.~6, pp.
  1965--1982, 2002.

\bibitem{fantuzzi2020bounding}
G.~Fantuzzi and D.~Goluskin, ``{Bounding Extreme Events in Nonlinear Dynamics
  Using Convex Optimization},'' \emph{SIAM Journal on Applied Dynamical
  Systems}, vol.~19, no.~3, pp. 1823--1864, 2020.

\bibitem{villani2008optimal}
C.~Villani, \emph{{Optimal Transport: Old and New}}.\hskip 1em plus 0.5em minus
  0.4em\relax Springer Science \& Business Media, 2008, vol. 338.

\bibitem{santambrogio2015optimal}
F.~Santambrogio, ``{Optimal Transport for Applied Mathematicians},''
  \emph{Birk{\"a}user, NY}, vol.~55, no. 58-63, p.~94, 2015.

\bibitem{peyre2019computational}
G.~Peyr{\'e}, M.~Cuturi \emph{et~al.}, ``{Computational Optimal Transport: With
  Applications to Data Science},'' \emph{Foundations and
  Trends{\textregistered} in Machine Learning}, vol.~11, no. 5-6, pp. 355--607,
  2019.

\bibitem{lewis1980relaxation}
R.~Lewis and R.~Vinter, ``Relaxation of optimal control problems to equivalent
  convex programs,'' \emph{Journal of Mathematical Analysis and Applications},
  vol.~74, no.~2, pp. 475--493, 1980.

\bibitem{henrion2008nonlinear}
D.~Henrion, J.~B. Lasserre, and C.~Savorgnan, ``Nonlinear optimal control
  synthesis via occupation measures,'' in \emph{2008 47th IEEE Conference on
  Decision and Control}.\hskip 1em plus 0.5em minus 0.4em\relax IEEE, 2008, pp.
  4749--4754.

\bibitem{henrion2013convex}
D.~Henrion and M.~Korda, ``{Convex Computation of the Region of Attraction of
  Polynomial Control Systems},'' \emph{IEEE TAC}, vol.~59, no.~2, pp. 297--312,
  2013.

\bibitem{korda2013inner}
M.~Korda, D.~Henrion, and C.~N. Jones, ``Inner approximations of the region of
  attraction for polynomial dynamical systems,'' \emph{IFAC Proceedings
  Volumes}, vol.~46, no.~23, pp. 534--539, 2013.

\bibitem{korda2014convex}
M.~Korda, D.~Henrion, and C.~Jones, ``{Convex Computation of the Maximum
  Controlled Invariant Set For Polynomial Control Systems},'' \emph{SIAM
  Journal on Control and Optimization}, vol.~52, no.~5, pp. 2944--2969, 2014.

\bibitem{prajna2004safety}
S.~Prajna and A.~Jadbabaie, ``{Safety Verification of Hybrid Systems Using
  Barrier Certificates},'' in \emph{International Workshop on Hybrid Systems:
  Computation and Control}.\hskip 1em plus 0.5em minus 0.4em\relax Springer,
  2004, pp. 477--492.

\bibitem{prajna2006barrier}
S.~Prajna, ``Barrier certificates for nonlinear model validation,''
  \emph{Automatica}, vol.~42, no.~1, pp. 117--126, 2006.

\bibitem{rantzer2004analysis}
A.~Rantzer and S.~Prajna, ``{On Analysis and Synthesis of Safe Control Laws},''
  in \emph{42nd Allerton Conference on Communication, Control, and
  Computing}.\hskip 1em plus 0.5em minus 0.4em\relax University of Illinois,
  2004, pp. 1468--1476.

\bibitem{miller2020recovery}
J.~{Miller}, D.~{Henrion}, and M.~{Sznaier}, ``{Peak Estimation Recovery and
  Safety Analysis},'' \emph{IEEE Control Systems Letters}, vol.~5, no.~6, pp.
  1982--1987, 2020.

\bibitem{miller2022distance_short}
J.~Miller and M.~Sznaier, ``{Bounding the Distance of Closest Approach to
  Unsafe Sets with Occupation Measures},'' in \emph{2022 IEEE 61st Conference
  on Decision and Control (CDC)}, 2022, pp. 5008--5013.

\bibitem{deza2009encyclopedia}
M.~M. Deza and E.~Deza, ``{Encyclopedia of Distances},'' in \emph{Encyclopedia
  of distances}.\hskip 1em plus 0.5em minus 0.4em\relax Springer, Berlin,
  Heidelberg, 2009, pp. 1--583.

\bibitem{garcia2021superresolution}
H.~García, C.~Hernández, M.~Junca, and M.~Velasco, ``Approximate
  super-resolution of positive measures in all dimensions,'' \emph{Applied and
  Computational Harmonic Analysis}, vol.~52, pp. 251--278, 2021.

\bibitem{llavona1986approximation}
J.~G. Llavona, \emph{Approximation of Continuously Differentiable
  Functions}.\hskip 1em plus 0.5em minus 0.4em\relax Elsevier, 1986.

\bibitem{tacchi2022convergence}
M.~Tacchi, ``{Convergence of Lasserre’s hierarchy: the general case},''
  \emph{Optimization Letters}, vol.~16, no.~3, pp. 1015--1033, 2022.

\bibitem{waki2006sums}
H.~Waki, S.~Kim, M.~Kojima, and M.~Muramatsu, ``{Sums of Squares and
  Semidefinite Programming Relaxations for Polynomial Optimization Problems
  with Structured Sparsity},'' \emph{SIOPT}, vol.~17, no.~1, pp. 218--242,
  2006.

\bibitem{wang2021tssos}
J.~Wang, V.~Magron, and J.-B. Lasserre, ``{TSSOS: A Moment-SOS hierarchy that
  exploits term sparsity},'' \emph{SIAM J. Optim.}, vol.~31, no.~1, pp. 30--58,
  2021.

\bibitem{riener2013exploiting}
C.~Riener, T.~Theobald, L.~J. Andr{\'e}n, and J.~B. Lasserre, ``{Exploiting
  Symmetries in SDP-Relaxations for Polynomial Optimization},''
  \emph{Mathematics of Operations Research}, vol.~38, no.~1, pp. 122--141,
  2013.

\bibitem{schlosser2020sparse}
C.~Schlosser and M.~Korda, ``Sparse moment-sum-of-squares relaxations for
  nonlinear dynamical systems with guaranteed convergence,'' 2020.

\bibitem{wang2021chordal}
J.~Wang, V.~Magron, and J.-B. Lasserre, ``{Chordal-TSSOS: A Moment-SOS
  Hierarchy That Exploits Term Sparsity with Chordal Extension},'' \emph{SIAM
  J. Optim.}, vol.~31, no.~1, pp. 114--141, 2021.

\bibitem{vandenberghe2015chordal}
L.~Vandenberghe, M.~S. Andersen \emph{et~al.}, ``{Chordal Graphs and
  Semidefinite Optimization},'' \emph{Foundations and Trends{\textregistered}
  in Optimization}, vol.~1, no.~4, pp. 241--433, 2015.

\bibitem{lasserre2006convergent}
J.~B. Lasserre, ``{Convergent SDP‐Relaxations in Polynomial Optimization with
  Sparsity},'' \emph{SIAM Journal on Optimization}, vol.~17, no.~3, pp.
  822--843, 2006.

\bibitem{tacchi2021thesis}
M.~Tacchi, ``Moment-sos hierarchy for large scale set approximation.
  application to power systems transient stability analysis,'' Ph.D.
  dissertation, Toulouse, INSA, 2021.

\bibitem{henrion2003gloptipoly}
D.~Henrion and J.-B. Lasserre, ``{GloptiPoly: Global Optimization over
  Polynomials with Matlab and SeDuMi},'' \emph{ACM Transactions on Mathematical
  Software (TOMS)}, vol.~29, no.~2, pp. 165--194, 2003.

\bibitem{lofberg2004yalmip}
J.~{Lofberg}, ``{YALMIP : a toolbox for modeling and optimization in MATLAB},''
  in \emph{ICRA (IEEE Cat. No.04CH37508)}, 2004, pp. 284--289.

\bibitem{mosek92}
M.~ApS, \emph{The MOSEK optimization toolbox for MATLAB manual. Version 9.2.},
  2020.

\bibitem{miller2021uncertain}
J.~Miller, D.~Henrion, M.~Sznaier, and M.~Korda, ``{Peak Estimation for
  Uncertain and Switched Systems},'' in \emph{2021 60th IEEE Conference on
  Decision and Control (CDC)}, 2021, pp. 3222--3228.

\bibitem{young1942generalized}
L.~C. Young, ``{Generalized Surfaces in the Calculus of Variations},''
  \emph{Annals of mathematics}, vol.~43, pp. 84--103, 1942.

\bibitem{yannakakis1991expressing}
M.~Yannakakis, ``{Expressing combinatorial optimization problems by Linear
  Programs},'' \emph{Journal of Computer and System Sciences}, vol.~43, no.~3,
  pp. 441--466, 1991.

\bibitem{gouveia2013lifts}
J.~Gouveia, P.~A. Parrilo, and R.~R. Thomas, ``{Lifts of Convex Sets and Cone
  Factorizations},'' \emph{Mathematics of Operations Research}, vol.~38, no.~2,
  pp. 248--264, 2013.

\bibitem{anderson1976discrete}
D.~Anderson and M.~Osborne, ``Discrete, linear approximation problems in
  polyhedral norms,'' \emph{Numerische Mathematik}, vol.~26, no.~2, pp.
  179--189, 1976.

\bibitem{anderson1987linear}
E.~J. Anderson and P.~Nash, \emph{Linear programming in infinite-dimensional
  spaces: theory and applications}.\hskip 1em plus 0.5em minus 0.4em\relax John
  Wiley \& Sons, 1987.

\bibitem{cimpric2011quadratic}
J.~Cimpri\u{c}, M.~Marshall, and T.~Netzer, ``{Closures of Quadratic
  Modules},'' \emph{Automatica}, vol. 183, no.~1, pp. 445--474, 2011.

\bibitem{putinar1993compact}
M.~Putinar, ``{Positive Polynomials on Compact Semi-algebraic Sets},''
  \emph{Indiana University Mathematics Journal}, vol.~42, no.~3, pp. 969--984,
  1993.

\bibitem{trnovska2005strong}
M.~Trnovsk\'{a}, ``{Strong Duality Conditions in Semidefinite Programming},''
  \emph{Journal of Electrical Engineering}, vol.~56, no.~12, pp. 1--5, 2005.

\end{thebibliography}

\end{document}